\documentclass[a4paper,twoside,10pt]{article}

\usepackage[a4paper,left=3cm,right=3cm, top=3cm, bottom=3cm]{geometry}
\usepackage[latin1]{inputenc}
\usepackage{mathrsfs}
\usepackage{graphicx}
\usepackage{epstopdf}
\usepackage{amsmath}
\usepackage{amsthm}
\usepackage{amssymb}
\usepackage{stmaryrd}
\usepackage{cite}
\usepackage{color}
\usepackage[font=footnotesize,labelfont=bf]{caption}
\usepackage[author={Lorenzo}]{pdfcomment}

\usepackage{algorithm}
\restylefloat{table}
\theoremstyle{plain}
\newtheorem{theorem}{Theorem}[section]

\newtheorem{lemma}[theorem]{Lemma}
\newtheorem{proposition}[theorem]{Proposition}
\theoremstyle{definition}

\theoremstyle{remark}
\newtheorem{remark}{Remark}

\newcommand{\red}[1]{{\color{red}{#1}}} 

\newcommand{\Norm}[1]{{\left\|{#1} \right\|}}
\newcommand{\SemiNorm}[1]{{\left|{#1} \right|}}
\newcommand{\jump}[1]{\left[\!\left[#1\right]\!\right]}

\newcommand{\p}{p}
\newcommand{\h}{h}
\newcommand{\taun}{\mathcal T_h}
\newcommand{\betabold}{\boldsymbol\beta}
\newcommand{\Pcalp}{\mathcal P_\p}
\newcommand{\T}{T}
\newcommand{\Nbb}{\mathbb N}
\newcommand{\Pbb}{\mathbb P}
\newcommand{\Rbb}{\mathbb R}

\newcommand{\Gammaplus}{\Gamma_+}
\newcommand{\Gammaminus}{\Gamma_-}
\newcommand{\xbf}{\mathbf x}
\newcommand{\nbf}{\mathbf n}
\newcommand{\qp}{q_\p}
\newcommand{\qpmo}{q_{\p-1}}
\newcommand{\qpmoT}{\qpmo^\T}
\newcommand{\qpFplusT}{\qp^{\FplusT}}
\newcommand{\Ihat}{\widehat I}

\newcommand{\Pizp}{\Pi^0_\p}

\newcommand{\vbf}{\mathbf v}

\newcommand{\vbfhat}{\widehat{\vbf}}
\newcommand{\That}{\widehat T}
\newcommand{\ptilde}{\widetilde \p}
\newcommand{\Qhat}{\widehat Q}
\newcommand{\FplusThat}{F^+_{\That}}
\newcommand{\FminusThat}{F^-_{\That}}
\newcommand{\Dcal}{\mathcal D}
\newcommand{\zbf}{\mathbf z}

\newcommand{\jbf}{\mathbf j}

\newcommand{\cbar}{\overline c}
\newcommand{\Fcalh}{\mathcal F_h}
\newcommand{\F}{F}
\newcommand{\hT}{\h_\T}
\newcommand{\hF}{\h_\F}
\newcommand{\nbfT}{\mathbf n_\T}
\newcommand{\nbfF}{\mathbf n_\F}
\newcommand{\FcalT}{\red{\mathcal F_\T}}
\newcommand{\Vh}{V_\h}
\newcommand{\uh}{u_\h}
\newcommand{\vh}{v_\h}
\DeclareMathOperator{\out}{out}
\DeclareMathOperator{\interior}{in}
\newcommand{\GammaTout}{\Gamma^\T_{\out}}
\newcommand{\GammaTin}{\Gamma^\T_{\interior}}
\newcommand{\Gammah}{\Gamma_\h^{\interior}}
\newcommand{\Tplus}{\T^+}
\newcommand{\Tminus}{\T^-}
\newcommand{\vhplus}{\vh^+}
\newcommand{\vhminus}{\vh^-}
\newcommand{\Bcal}{\mathcal B}
\newcommand{\nablah}{\nabla_\h}
\newcommand{\nbfGamma}{\nbf_{\Gamma}}
\DeclareMathOperator{\DG}{DG}
\newcommand{\FplusT}{\F_\T^+}
\newcommand{\eh}{e_\h}
\DeclareMathOperator{\REM}{\mathcal R}
\newcommand{\J}{J}
\newcommand{\jtilde}{\widetilde j}

\title{\normalsize{$\h\p$-optimal convergence of the original DG method for linear hyperbolic problems on special simplicial meshes}}
\author{\normalsize{Z. Dong\thanks{Inria, 2 rue Simone Iff, 75589 Paris, France and CERMICS, Ecole des Ponts, 77455 Marne-la-Vall\'{e}e 2, France, {\tt zhaonan.dong@inria.fr}},
L. Mascotto\thanks{Dipartimento di Matematica e Applicazioni, Universit\`a di Milano Bicocca, 20125 Milan, Italy, {\tt lorenzo.mascotto@unimib.it};
IMATI-CNR, 27100, Pavia, Italy;
Fakult\"at f\"ur Mathematik, Universit\"at Wien, 1090 Vienna, Austria, {\tt lorenzo.mascotto@univie.ac.at}}}}
\date{}

\begin{document}

\maketitle

\begin{abstract}
\noindent We prove $\h\p$-optimal error estimates for the original {discontinuous Galerkin (DG)} method
when approximating solutions to first-order hyperbolic problems
with constant convection fields
in the $L^2$ and DG norms.
The main theoretical tools used in the analysis
are novel $\h\p$-optimal approximation properties
of the special projector introduced in [Cockburn, Dong, Guzm\'an, SINUM, 2008].
We assess the theoretical findings on some test cases.


\end{abstract}

\section{Introduction} \label{section:introduction}

The discontinuous Galerkin ({DG}) method was
{first} introduced in~\cite{Reed-Hill:1973}
for the approximation of solutions to the neutron transport equation,
i.e., a first-order linear hyperbolic problem.
This method,
{which we shall refer to as the original DG method,}
was {later} analysed in~\cite{Lesaint-Raviart:1974};
advances on the convergence of the scheme were given in~\cite{Johnson-Pitkaranta:1986}.
The convergence rates in the $L^2$ norm of the $h$-version of the method
proven therein are suboptimal by half an order;
this was also apparent from the numerical experiments in~\cite{Peterson:1991}.
On special classes of meshes,
optimal convergence for the $\h$-version was established in~\cite{Cockburn-Dong-Guzman:2008}.

Fewer results are available for the $\p$- and $\h\p$-versions of the method.
In~\cite{Houston-Schwab-Suli:2000},
$\h\p$ error estimates were derived
for a streamline diffusion version of the discontinuous Galerkin method
on fairly general quadrilateral meshes.
{Shortly after, in~\cite{Houston-Schwab-Suli:2002},}
for the original method
and linear polynomial convection fields,
$\p$ convergence was shown suboptimal by half a order for the convection-diffusion case,
but optimal for the linear hyperbolic case;
for more general vector fields, suboptimality by one order and a half was discussed as well.
The main theoretical tools in this reference
are $\h\p$ optimal approximation properties of the $L^2$ projector
on the boundary of tensor product elements.
More recently~\cite{Dong-Mascotto:2021} and again for tensor product elements,
such a suboptimality was reduced to half an order only for a special class of convective fields.

This paper aims at extending even more the knowledge on the convergence analysis
of the $\h\p$-{discontinuous Galerkin} ($\h\p$-DG)
for first order hyperbolic problems in the following aspects:
for constant convection fields,
\begin{itemize}
    \item we prove $\h\p$-optimal convergence of the method in the $L^2$ norm,
    thus generalizing the results in~\cite{Cockburn-Dong-Guzman:2008}
    to the $\p$-version of the method;
    \item we prove $\h\p$-optimal convergence of the method in the DG norm
    also on special simplicial meshes as in~\cite{Cockburn-Dong-Guzman:2008},
    thus generalizing the results in~\cite{Houston-Schwab-Suli:2000}
    to the case of simplicial meshes.
\end{itemize}
To {this} aim, we analyse the $\h\p$ approximation properties of the {Cockburn-Dong-Guzm\'an}
(henceforth denoted by CDG) projector
introduced in~\cite{Cockburn-Dong-Guzman:2008B}.
In particular, we generalise the one dimensional results in~\cite[Lemmas~$3.5$ and~$3.6$]{Schotzau-Schwab:2000} to simplices in two and three dimensions.

A possible reason why the simplicial mesh case was not contemplated in~\cite{Houston-Schwab-Suli:2000, Houston-Schwab-Suli:2002}
is that $\h\p$-optimal convergence estimates
for the trace of the polynomial $L^2$-projection operator on simplices{,}
crucial in the analysis of the CDG projector properties
were derived later~\cite{Chernov:2012, Melenk-Wurzer:2014}.

\paragraph*{Notation.}
Let~$D$ be a Lipschitz domain in~$\Rbb^d$, $d=1$, $2$, and $3$, with boundary~$\partial D$.
The space of Lebesgue measurable and square integrable functions over~$D$ is~$L^2(D)$.
The Sobolev space of positive integer order~$s$ is~$H^s(D)$.
We also write~$L^2(D) = H^0(D)$.
We endow~$H^s(D)$ with the inner product, seminorm, and norm
\[
(\cdot,\cdot)_{s,D},
\qquad\qquad\qquad
\SemiNorm{\cdot}_{s,D},
\qquad\qquad\qquad
\Norm{\cdot}_{s,D}.
\]
Interpolation theory is used to construct Sobolev spaces of positive noninteger order;
duality is used to define negative order Sobolev spaces.

The trace theorem is valid for~$H^s(D)$, $1/2<s<3/2$.
In particular, given~$s$ in the above range and~$g$ in~$H^{s-\frac12}(\Gamma_g)$,
being $\Gamma_g$ any subset of~$\partial D$ with nonzero measure in~$\partial D$,
we are allowed to define the space
\[
H^s_g(D,\Gamma_g):=
\{ v \in H^s(D,\Gamma_D) \mid v_{|\Gamma_D} = g \}.
\]
The space of polynomials of nonnegative degree~$\p$ over~$D$ is~$\Pbb_\p(D)$.

\paragraph*{The continuous problem.}
Given~$\Omega$ a Lipschitz domain in~$\Rbb^d$, $d=1,2,3$, with boundary~$\Gamma$,
consider~$\betabold$ in~$\Rbb^d$ and $c$ in~$L^{\infty}(\Omega)$.
We introduce the inflow part~$\Gammaminus$ of the boundary of~$\Omega$ as follows:
given~$\nbfGamma(\xbf)$ the outward normal to~$\Gamma$ at~$\xbf$,
\[
\Gammaminus := \{ \xbf \in \partial \Omega \mid \betabold \cdot \nbfGamma(\xbf) < 0 \}.
\]
{The characteristic part~$\Gamma_0$ is analogously defined as
\[
\Gamma_0 := \{ \xbf \in \partial \Omega \mid \betabold \cdot \nbfGamma(\xbf) = 0 \}.
\]
The outflow part~$\Gammaplus$ of~$\Gamma$ is given by~$\Gamma\setminus(\Gammaminus \cup  \Gamma_0)$.}
We are interested in the approximation of solutions to convection-reaction problems:
given~$g$ in~$L^2(\Gammaminus)$,
\begin{equation} \label{strong-formulation}
\begin{cases}
\text{find } u \text{ such that } \\
\betabold \cdot \nabla u + c \ u = f  & \text{in } \Omega \\
u = g                                 & \text{on } \Gammaminus.
\end{cases}
\end{equation}
Introduce the graph space
\[
V :=
\{ v \in L^2 (\Omega) \mid \betabold \cdot \nabla v \in L^2 (\Omega) \},
\]
which we endow with the graph norm
\begin{equation} \label{graph-norm}
\Norm{v}_{V}^2
:= \Norm{v}_{0,\Omega}^2 + \Norm{\betabold \cdot \nabla v}_{0,\Omega}^2 .
\end{equation}
We can define a trace operator from the graph space onto the space
\[
L^2(\vert \betabold \cdot \nbfGamma \vert, \Gamma)
:= \{ v \text{ measurable on } \Gamma \mid
        (\vert \betabold \cdot \nbfGamma \vert v,v)_{0,\Gamma} < \infty \}.
\]
{Let
\begin{equation} \label{graph-space-bcs}
V_g := \{ v \in V \mid \text{$g$ is the trace of $v$ and belongs to }
            L^2(\vert \betabold \cdot \nbfGamma \vert, \Gamma) \}.
\end{equation}}
We introduce the bilinear form on~$V_g \times L^2(\Omega)$ as
\[
b(u,v) = (\betabold \cdot \nabla u, v)_{0,\Omega}
          + (c\ u, v)_{0,\Omega}.
\]
The weak formulation of problem~\eqref{strong-formulation} reads
\begin{equation} \label{weak-formulation}
\begin{cases}
\text{find } u \in V_g \text{ such that} \\
b(u,v) = (f,v)_{0,\Omega} \qquad \forall v \in L^2(\Omega).
\end{cases}
\end{equation}
Henceforth, we assume that there exists a positive constant~$\cbar_0$ such that
\begin{equation} \label{standard-assumption-convection}
\cbar :=  c - \frac12 \nabla\cdot\betabold \ge \cbar_0 .
\end{equation}
Problem~\eqref{weak-formulation} is well posed with respect to the graph norm in~\eqref{graph-norm};
see, e.g.,~\cite{Bardos:1970}.

\paragraph*{Structure of the paper.}
In Section~\ref{section:meshes-method},
we introduce admissible simplicial meshes in the sense of~\cite{Cockburn-Dong-Guzman:2008}
and recall the original DG method from~\cite{Reed-Hill:1973}.
The main technical result of the paper,
i.e., $\h\p$ approximation properties of the CDG projector
are derived in Section~\ref{section:main}
in one, two, and three dimensions.
Such estimates are used to derive $\h\p$-optimal convergence of the method in the $L^2$ norm in Section~\ref{section:convergence};
there, we also show $\h\p$-optimal convergence of the method in a DG norm.
We assess the theoretical results with several numerical experiments
in Section~\ref{section:numerical-experiments}
and draw some conclusions in Section~\ref{section:conclusions}.

\section{Admissible meshes and the method} \label{section:meshes-method}

We introduce admissible simplicial meshes for the forthcoming analysis,
and define associated Sobolev and polynomial broken spaces
in Section~\ref{subsection:meshes},
and recall the original DG method for problem~\eqref{weak-formulation}
in Section~\ref{subsection:method}.

\subsection{Admissible simplicial meshes} \label{subsection:meshes}
We follow~\cite[Section~1]{Cockburn-Dong-Guzman:2008} and introduce admissible simplicial meshes
that are instrumental in deriving the main result of the paper;
see Section~\ref{section:main} below.

{Let~$\taun$ be a simplicial mesh of the domain~$\Omega$}
and~$\Fcalh$ be its set of ($d-1$)-dimensional facets.
We distinguish the facets in~$\Fcalh$ into internal and boundary facets;
the former are those~$\F$ not contained in~$\partial\Omega$.
The union of the interior facets of~$\taun$ is~$\Gammah$.

With each element~$\T$ of~$\taun$ we associate its diameter~$\hT$
and its outward unit vector~$\nbfT$,
which is defined almost everywhere on~$\partial\T$.
The set of ($d-1$)-dimensional facets of~$\T$ is~$\FcalT$.
With each facet~$\F$ in~$\Fcalh$ we associate its diameter~$\hF$
and a unit normal vector~$\nbfF$,
which is pointing outward~$\Omega$ if~$\F$ is a boundary facet.
{We only consider meshes~$\taun$ that are $\sigma$ shape-regular
i.e., any element~$\T$ of~$\taun$ is star-shaped
with respect to a ball of radius
larger than or equal to $\hT \sigma$.}
{We associate with each mesh~$\taun$
piecewise differential operators by adding a subscript~$\h$,
for instance~$\nablah$ denotes the piecewise gradient.}

We distinguish the facets in~$\FcalT$ into outflow, inflow, and characteristics (with respect to~$\betabold$) facets,
depending on whether they satisfy either of the two following properties:
\[
\nbfT{}_{|\F}\cdot\betabold > 0,
\qquad\qquad\qquad
\nbfT{}_{|\F}\cdot\betabold < 0,
\qquad\qquad\qquad
\nbfT{}_{|\F}\cdot\betabold = 0.
\]
The union of the outflow and inflow facets is~$\GammaTout$ and~$\GammaTin$, respectively.

We require that each element~$\T$ of~$\taun$ satisfy the two following properties:
\begin{itemize}
\item[(\textbf{A1})] $\T$ has only one outflow facet,
{which we denote by~$\FplusT$};
\item[(\textbf{A2})] each interior facet~$\F$ that is an inflow facet for~$\T$
is included in the outflow facet for another simplex in the mesh.
\end{itemize}
The second assumption above implies that~$\taun$ can be nonconforming,
i.e., hanging facets are allowed.
Meshes satisfying the above properties can be always constructed;
see, e.g., \cite[Appendix]{Cockburn-Dong-Guzman:2008}.

The focus of the paper is on $\p$-optimal estimates;
therefore, we pick quasi-uniform meshes
and denote the mesh size of~$\taun$ by~$\h$.

Given~$\p$ in~$\Nbb$ and~$s$ in~$\Rbb^+$,
we associate with a mesh~$\taun$ as in Section~\ref{subsection:meshes}
the spaces
\[
\Pbb_\p(\taun)
:= \{ \qp \in L^2(\Omega) \mid \qp{}_{|\T} \in \Pbb_\p(\T) \quad \forall \T \in \taun \}
\]
and
\[
H^s(\taun) := \{ v \in L^2(\Omega) \mid v_{|\T} \in H^s(\T)\quad \forall \T \in \taun \} .
\]
Differential operators defined piecewise over~$\taun$ are denoted with the same symbol of the original operator with an extra subcript~$\h$;
for instance, the broken gradient is~$\nablah$.

Given an interior facet~$\F$, we denote by~$\Tplus$ and~$\Tminus$
the two elements of~$\taun$
such that~$\F$ is an outflow and inflow facet for~$\Tplus$ and~$\Tminus$, respectively.
Given~$\vh$ in~$\Pbb_\p(\taun)$ and~$\xbf$ in the face~$\F$,
we write
\[
\vh^{\pm}(\xbf)  = \lim_{\delta\downarrow 0} \vh(\xbf \pm \delta \betabold).
\]
We introduce the jump operator~$\jump{\cdot}:H^1(\taun) \to L^2(\Gammah)$ given by $\jump{\vh}{}_{|\F}:=\vhplus-\vhminus$
for all internal edges~$\F$.

{Given a boundary facet~$\F$, we denote by~$\T$ the only
element of~$\taun$ such that~$\F$ belongs to~$\FcalT$.
Since every function~$\vh$ in~$\Pbb_\p(\taun)$
is single valued on boundary facets,
for all~$\xbf$ in the face~$\F$,
we write $\vh(\xbf)= \vh{}_{\T}(\xbf)$,
where~$\T$ is the only element of~$\taun$
such that~$\F$ belongs to~$\FcalT$.}

Henceforth, given two positive quantities~$a$ and~$b$,
we write~$a\lesssim b$ is there exists a positive constant~$c$
possibly depending on the shape-regularity parameter~$\sigma$ only,
such that~$a \le c \ b $.
If~$a\lesssim b$ and~$b \lesssim a$ at once, we write~$a \approx b$.

\subsection{The original DG method} \label{subsection:method}
Consider the space~$\Vh$ given by~$\Pbb_\p(\taun)$.
The original~\cite{Reed-Hill:1973} DG method for~\eqref{weak-formulation} reads
\begin{equation} \label{method}
\begin{cases}
\text{find } \uh \in \Vh \text{ such that} \\
\Bcal(\uh,\vh)
= (f,\vh)_{0,\Omega} - (g,\nbfGamma \cdot \betabold \ \vh)_{0,\Gammaminus}
\qquad \forall \vh \in \Vh.
\end{cases}
\end{equation}
where
\[
\Bcal(\uh,\vh)
:=  (\betabold\cdot \nablah\uh + c\ \uh, \vh)_{0,\Omega}
    - \sum_{\T \in \taun} (\jump{\uh}, \nbfT \cdot \betabold \ \vhplus)_{0,\GammaTin}
    -  (\uh, \nbfGamma \cdot \betabold \ \vh)_{0,\Gammaminus}.
\]
The solution~$\uh$ can be computed starting at the inflow boundary~$\Gammaminus$;
{solving local problems derived from~\eqref{method}
on the elements abutting~$\Gammaminus$;
transmitting the solution to the neighbouring elements
in the $\betabold$ direction through upwind.}

An equivalent alternative formulation used, e.g., in~\cite{Cockburn-Dong-Guzman:2008} is derived
using an integration by parts and the fact that~$\betabold$ is a constant field.
Notably, an integration by parts implies that
the bilinear form~$\Bcal(\cdot,\cdot)$ can be rewritten as
\small{\begin{equation} \label{Cockburn-formulation}
\begin{split}
\Bcal(\uh,\vh)
& = - (\uh, \betabold\cdot\nablah \vh)_{0,\Omega}
  \!+\! \sum_{\T\in\taun} (\uh^- , \nbfT \cdot \betabold \
         \jump{\vh})_{0,\GammaTin \setminus\Gammaminus}
  \!+\! (c\ \uh, \vh)_{0,\Omega}
  + (\uh,\nbfGamma\cdot \betabold \ \vh)_{0,\Gammaplus}.
\end{split}
\end{equation}}\normalsize{}
Method~\eqref{method} is well posed
{and coincides with that introduced in~\cite{Reed-Hill:1973}
and analysed in~\cite{Lesaint-Raviart:1974}.}
A priori estimates with respect to the data are derived in~\cite[Lemma~2.4 and Section~5]{Houston-Schwab-Suli:2000}.
{Previous results requiring stronger assumptions
on~$c$ are given, e.g., in \cite{Johnson-Pitkaranta:1986}.}

\begin{proposition} \label{proposition:stability-method}
Let~$\cbar$ be given in~\eqref{standard-assumption-convection}.
The following bound holds true:
\[
\cbar \Norm{\uh}_{0,\Omega}^2
+ \sum_{\T \in \taun} \Norm{\vert\nbfT \cdot \betabold \vert^\frac12\jump{\uh}}_{0,\GammaTin}^2
+ \Norm{\vert\nbfGamma \cdot \betabold \vert^{\frac12}\uh}_{0,\Gamma}^2
\le \cbar^{-1} \Norm{f}_{0,\Omega} + 2 \Norm{g}_{0,\Gammaminus}^2.
\]
\end{proposition}
All constants appearing in the stability estimates of Proposition~\ref{proposition:stability-method} are explicit,
and independent of~$\h$ and~$\p$.
This is essential to derive the optimal error estimates
in Section~\ref{section:convergence} below.
Former stability results display stability estimates
that are suboptimal in the polynomial degree;
see, e.g., \cite[Theorem~2.1]{Johnson-Pitkaranta:1986}.

Method~\eqref{method} is consistent.
The following Galerkin orthogonality property follows:
given~$u$ and~$\uh$ the solutions to~\eqref{weak-formulation} and~\eqref{method},
\begin{equation} \label{Galerkin-orthogonality}
\Bcal(u-\uh,\vh) = 0
\qquad\qquad\qquad
\forall \vh \in \Vh.
\end{equation}
In particular, the bilinear form~$\Bcal(\cdot,\cdot)$
is also well defined on $V_g \times \Vh$,
{where~$V_g$ is defined in~\eqref{graph-space-bcs}.}

\begin{remark} \label{remark:beta-non-constant}
{The case of nonconstant~$\betabold$
is considered in~\cite{Cockburn-Dong-Guzman-Qian:2010}.
Further assumptions on the advection field are needed.
The presentation of this manuscript might be generalised
to a similar setting at the price of extra technicalities.}
\end{remark}

\section{The CDG projector and its approximation properties} \label{section:main}

We recall some notation and technical results from the theory of orthogonal polynomials in Section~\ref{subsection:preliminary-results};
we define of the {Cockburn-Dong-Guzm\'an} (CDG) projector
and state the main result in Section~\ref{subsection:CDG-projector};
we prove $\p$-optimal approximation properties of the CDG operator
in one, two, and three dimensions
in Sections~\ref{subsection:1D-approx}, \ref{subsection:2D-approx}, and~\ref{subsection:3D-approx}, respectively;
in Section~\ref{subsection:trace-estimates} we prove trace-type estimates for the CDG projectors.

\subsection{Preliminary results} \label{subsection:preliminary-results}
We recall $\h\p$ approximation properties of the $L^2$ projector on a simplex~$\T$ in the $L^2$ norm on the boundary of~$\T$.
This result traces back to~\cite[Theorem~2.1]{Chernov:2012}
and~\cite[Theorem~1.1]{Melenk-Wurzer:2014}.
Let~$\Pizp:L^2(\T) \to \Pbb_{\p}(\T)$ denote the $L^2$ projector defined as
\begin{equation} \label{L2-projector}
(v-\Pizp v, \qp)_{0,\T} = 0
\qquad\qquad\qquad \forall \qp \in \Pbb_\p(\T).
\end{equation}

\begin{lemma} \label{lemma:Chernov-Melenk}
Under the notation of Section~\ref{subsection:meshes},
there exists a positive~$C$ independent of~$\sigma$ and~$d$ such that
\[
\Norm{u - \Pizp u}_{0,\FplusT} \le
C \left(\frac{\hT}{\p} \right)^{\frac12} \SemiNorm{u}_{1,\T}
\qquad\qquad\qquad
\forall u \in H^1(\T).
\]
As a consequence, for a possibly larger constant~$C$
{and~$k$ positive,} we also have
\[
\Norm{u - \Pizp u}_{0,\FplusT} \le
C \left(\frac{\hT}{\p} \right)^{k+\frac12} \Norm{u}_{k+1,\T}
\qquad\qquad\qquad
\forall u \in H^{k+1}(\T).
\]
In the one dimensional case, the norm on the right-hand side of
the last inequality is in fact a seminorm.
\end{lemma}


\noindent As we are interested in deriving estimates explicit in the polynomial degree
and the estimates explicit in the element size are already known from~\cite{Cockburn-Dong-Guzman:2008B},
we shall prove approximation properties on reference elements.

Set~$\Ihat:=[-1,1]$ the reference interval.
We recall some properties of orthogonal polynomials;
see, e.g., \cite{Shen-Tang-Wang:2011}.
Let~$\{ L_j \}_{j=0}^{+\infty}$ be the $L^2$ orthogonal set of the Legendre polynomials over~$\Ihat$ satisfying
\cite[Corollary~3.6, $\alpha=0$, $\beta=0$]{Shen-Tang-Wang:2011}
\begin{equation} \label{orthogonality-Legendre}
(L_i,L_j)_{0,\Ihat} = \frac{2}{2j+1} \delta_{i,j}
\qquad\qquad
\forall i,j \in \Nbb_0 .
\end{equation}
Given~$\Gamma(\cdot)$ the Gamma function,
further let~$\{ \J_j^{\ell} \}_{j=0}^{+\infty}$, $\ell > -1$,
be the weighted-$L^2$ orthogonal set of Jacobi polynomials over~$\Ihat$ satisfying
\cite[Corollary~3.6, $\alpha=\ell$, $\beta=0$]{Shen-Tang-Wang:2011}
\begin{equation} \label{orthogonality-Jacobi}
((1-x)^\ell \J_i^\ell,\J_j^\ell)_{0,\Ihat}
= \frac{2^{\ell+1}}{2j+\ell+1}
    \frac{\Gamma(j+\ell+1) \Gamma(\ell+1)}{\Gamma(\ell+1) \Gamma(j+\ell+1)}\delta_{j,\ell}
= \frac{2^{\ell+1}}{2j+\ell+1} \delta_{j,\ell}
\qquad\qquad
\forall i,j \in \Nbb_0 .
\end{equation}

\subsection{The CDG projector and the main result} \label{subsection:CDG-projector}
The CDG operator was introduced
for the analysis of superconvergent DG methods for second-order elliptic problems
in~\cite{Cockburn-Dong-Guzman:2008B}.
It can be defined for simplices in any dimensions
satisfying assumptions (\textbf{A1})--(\textbf{A2}).
In particular, recall that~$\FplusT$ is the only outflow facet of~$T$.

The CDG operator
$\Pcalp : H^{\frac12+\varepsilon}(\T) \to \Pbb_\p(\T)$
is defined as
\begin{subequations}
\label{definition:Pcalp}
\begin{align}
    & (v-\Pcalp v, \qpmoT)_{0,\T} = 0
    \qquad\qquad\qquad \forall \qpmoT \in \Pbb_{\p-1}(\T),
    \label{definition:Pcalp1}\\
    & (v-\Pcalp v, \qpFplusT)_{0,\FplusT} = 0
    \qquad\qquad\qquad \forall \qpFplusT \in \Pbb_{\p}(\FplusT)
    \label{definition:Pcalp2}.
\end{align}
\end{subequations}
In~\cite{Cockburn-Dong-Guzman:2008B}, the following error estimate was proven
for sufficiently smooth functions:
\[
\Norm{v-\Pcalp v}_{0,\T}
\le C \ \h^{\p+1} \SemiNorm{v}_{\p+1,\Omega} .
\]
The constant~$C$ above depends on~$\sigma$ and the polynomial degree~$\p$.
The remainder of the section is devoted to carefully detail the dependence on~$\p$.
Namely, we shall prove the following result.

\begin{theorem} \label{theorem:hp-approximation-CDG}
Let the assumptions (\textbf{A1}) and (\textbf{A2}) be valid,
$\T$ be in~$\taun$,
and {$\Pizp$ be the operator in~\eqref{L2-projector}.}
{Then, there exist a positive constant~$C$
independent of~$\h$ and~$\p$, such that,
for any~$v$ in $H^{k+1}(T)$, $k$ positive,}
we have
\[
\Norm{v-\Pcalp v}_{0,\T}
{\le C \frac{\h^{\min(\p,k)+1}}{\p^{k+1}} \Norm{v}_{k+1,\T}.}
\]
\end{theorem}

\subsection{The 1D case} \label{subsection:1D-approx}
We prove Theorem~\ref{theorem:hp-approximation-CDG} in the case of one dimensional simplices,
i.e., on intervals.
This was already done in~\cite[Lemmas 3.5 and 3.6]{Schotzau-Schwab:2000}.
However, we deem that reviewing the main steps from~\cite{Schotzau-Schwab:2000}
is beneficial for the understanding of the extension to the two and three dimensional cases
in Sections~\ref{subsection:2D-approx} and~\ref{subsection:3D-approx} below, respectively.

Without loss of generality, we pick~$\T = \Ihat$ and assume $-1$ to be the outflow facet.
In 1D, the operator~$\Pcalp$ is defined imposing orthogonality up to order~$\p-1$ over~$\Ihat$
and imposing that~$\Pcalp v (-1) = v(-1)$;
this last condition replaces the orthogonality on the outflow facet.

The proof of optimal $\p$-convergence can be split in the following steps:
\begin{enumerate}
    \item we write~$\Pcalp v$ as a combination of Legendre polynomials
    and compute the corresponding coefficients;
    up to order~$\p-1$, the coefficients are the same as those of the expansion of~$v$;
    the coefficient of order~$\p$ is given by a tail of the other Legendre coefficients;
   \item we relate such a tail to the coefficients of~$v$ minus its $L^2$ projection
   on the triangle restricted to the outflow facet;
    \item we deduce the assertion using trace error estimates as in Lemma~\ref{lemma:Chernov-Melenk}.
\end{enumerate}

\paragraph*{Step~1: writing~$\Pcalp v$ with respect to the Legendre basis and with explicit coefficients.}
Given~$v$ in~$L^2(\Ihat)$, we consider its expansion with respect to Legendre polynomials
\[
v(x) = \sum_{j=0}^{+\infty} \vbf_j \ L_j(x).
\]
Consider also the truncated expansion up to order~$\p$ given by
\[
\Pcalp v(x) = \sum_{j=0}^\p \vbfhat_j \ L_j(x).
\]
The orthogonality condition~\eqref{definition:Pcalp1} with respect to polynomials of degree~$\p-1$ gives
\[
\Pcalp v(x) = \sum_{j=0}^{\p-1} \vbf_j \ L_j(x) + \vbfhat_\p \ L_\p(x).
\]
The last coefficient is given imposing the condition $v(-1) = \Pcalp v(-1)$.
More precisely, since~$L_j(-1)=(-1)^j$ for all~$j$ in~$\Nbb$,
we write
\[
v(-1)=\sum_{j=0}^{+\infty} (-1)^j \vbf_j ,
\qquad\qquad
\Pcalp v(-1)=\sum_{j=0}^{\p-1} (-1)^j \vbf_j
          + (-1)^\p \vbfhat_\p.
\]
We deduce
\[
(-1)^\p \vbfhat_\p
= \sum_{j=\p}^{+\infty} \vbf_j (-1)^j ,
\]
whence we get
\begin{equation} \label{expansion:Pcalp-1D}
\Pcalp v(x) = \sum_{j=0}^{\p-1} \vbf_j \ L_j(x)
                {+} \big(\sum_{j=\p}^{+\infty}  (-1)^{j-\p}\vbf_j \big) \ L_\p(x).
\end{equation}
Given~$\Pizp$ the~$L^2$ projection as in~\eqref{L2-projector},
we arrive at
\begin{equation} \label{expansion:v-Pcalpv:1D}
\begin{split}
(v - \Pcalp v)(x)
& = \sum_{j=\p}^{+\infty} \vbf_j \ L_j(x)
    - \big(\sum_{j=\p}^{+\infty} (-1)^{j-\p} \vbf_j \big) \ L_\p(x) \\
& = \sum_{j=\p+1}^{+\infty} \vbf_j \ L_j(x)
    + \vbf_\p L_\p(x) - \vbf_\p L_\p(x)
    - \big(\sum_{j=\p+1}^{+\infty} (-1)^{j-\p} \vbf_j \big) \ L_\p(x) \\
& = \sum_{j=\p+1}^{+\infty} \vbf_j \ L_j(x)
    - \big(\sum_{j=\p+1}^{+\infty} (-1)^{j-\p} \vbf_j \big) \ L_\p(x) \\
& = (v - \Pizp v)(x) - (-1)^{-\p} \big(\sum_{j=\p+1}^{+\infty} (-1)^{j} \vbf_j \big) \ L_\p(x) .
\end{split}
\end{equation}

\paragraph*{Step~2: relate the tail of the Legendre coefficients to the trace of~$v-\Pizp v$
on the outflow facet.}
Observe that
\begin{equation} \label{trace:v-Pizpv:1D}
(v-\Pizp v) (x)
= \sum_{j=\p+1}^{+\infty} \vbf_j L_j(x),
\qquad
(v-\Pizp v) (-1)
= \sum_{j=\p+1}^{+\infty} \vbf_j L_j(-1)
= \sum_{j=\p+1}^{+\infty} (-1)^j \vbf_j.
\end{equation}
Combining~\eqref{expansion:v-Pcalpv:1D} and~\eqref{trace:v-Pizpv:1D} gives
\[
(v - \Pcalp v)(x)
= (v - \Pizp v)(x)
- (-1)^\p [(v-\Pizp v)(-1)] L_\p(x)
\]
Recall from~\eqref{orthogonality-Legendre} that
\begin{equation} \label{norm-Legendre}
\Norm{L_\p}^2_{0,\Ihat} = \frac{2}{2\p+1}.
\end{equation}
We deduce
\[
\Norm{v-\Pcalp v}_{0,\Ihat}^2
\le \Norm{v-\Pizp v}_{0,\Ihat}^2
     + \vert (v-\Pizp v)(-1) \vert^2 \frac{2}{2\p+1}.
\]

\paragraph*{Step~3: deduce the desired bound using Lemma~\ref{lemma:Chernov-Melenk}.}
The 1D version of Lemma~\ref{lemma:Chernov-Melenk} reads
\[
\vert (v-\Pizp v)(-1) \vert^2
\lesssim (\p+1)^{-1} \SemiNorm{v}_{1,\Ihat}^2.
\]
Combining the two equations above yields
\[
\Norm{v-\Pcalp v}_{0,\Ihat}
\lesssim \Norm{v-\Pizp v}_{0,\Ihat}
     + (\p+1)^{-1} \SemiNorm{v}_{1,\Ihat}.
\]
{The assertion follows
noting that~$\Pcalp$ and~$\Pizp$ preserve polynomials of maximum degree~$\p$
and standard polynomial approximation results.}

\subsection{The 2D case} \label{subsection:2D-approx}
We prove Theorem~\ref{theorem:hp-approximation-CDG} in two dimensions.
Compared to the 1D case in Section~\ref{subsection:1D-approx},
the role of orthogonal basis functions is now played by tensor product Jacobi polynomials (with different weights)
collapsed on the triangle via the Duffy transformation;
in other words, we use Koornwinder polynomials
following the construction in~\cite{Beuchler-Schoeberl:2006, Chernov:2012};
see also~\cite{Dubiner:1991, Koornwinder:1975}.

First, we introduce a reference 2D simplex~$\T$:
\begin{equation} \label{reference-triangle}
\That:=
\left\{
\left( z\frac{1-y}{2} , y    \right) \in \Rbb^2
\ \middle| \
(z,y)\in [-1,1]^2
\right\}.
\end{equation}
The reference triangle~$\That$ corresponds to the reference square~$\widehat Q:= [-1,1]^2$
collapsed under the Duffy transformation
\[
(z,y) \in \widehat Q  \qquad \Longrightarrow \qquad \left( z\frac{1-y}{2} , y \right) .
\]
This particular choice of~$\That$ is convenient for our purposes,
as we shall analyse the approximation properties of~$\Pcalp$
assuming that~$\FplusThat$ is the segment~$[-1,1] \times \{-1\}$.

We introduce the orthogonal polynomial basis (known as Koornwinder polynomial basis~\cite{Koornwinder:1975}) over~$\That$ given by
\begin{equation} \label{Koornwinder-basis}
\Phi_{j,\ell}(x,y)
:= L_j \left( \frac{2x}{1-y} \right) \J_\ell^{2j+1} (y) \left( \frac{1-y}{2} \right)^j
\qquad\qquad \forall j,\ \ell \in \Nbb .
\end{equation}
This is an orthogonal basis over~$\That$.
In fact, using the Duffy transformation, we have
\small{\[
\int_{\That} \Phi_{i,j}(x,y) \Phi_{k, \ell} (x,y) dx \ dy
= \left(\int_{-1}^1 L_i(z) L_k(z) \ dz\right)
    \left( \int_{-1}^1 \J_j^{2i+1} (y) \J_\ell^{2k+1} (y) \left( \frac{1-y}{2} \right)^{i+k+1} dy \right).
\]}\normalsize{}
If~$i\ne k$, the above inner product is zero.
Thus, without loss of generality, we can assume~$i=k$.
Using~\eqref{orthogonality-Jacobi},
we write
\small{\begin{equation} \label{orthogonality:Koornwinder:2D}
\begin{split}
& \int_{\That} \Phi_{i,j}(x,y) \Phi_{i, \ell}(x,y) dx \ dy
  = \left( \int_{-1}^1 L_i(z) L_i(z) \ dz \right)
    \left( \int_{-1}^1 \J_j^{2i+1} (y) \J_\ell^{2i+1} (y) \left( \frac{1-y}{2} \right)^{2i+1} dy \right) \\
& = {\frac{2}{2i+1}}  2^{-2j-1}
    \frac{2^{2j+2}}{2\ell+2j+2} \delta_{j,\ell}
  = {\frac{2}{2i+1}} 2^{-2j-1}
    \frac{2^{2j+1}}{j+\ell+1} \delta_{j,\ell}
  = {\frac{2}{2i+1}}
    \frac{1}{j+\ell+1} \delta_{j,\ell} .
\end{split}
\end{equation}}\normalsize{}
With this at hand, we extend the results in Section~\ref{subsection:1D-approx} to 2D simplices.
Henceforth, let~$v$ denote a sufficiently regular function over~$\That$.

\paragraph*{Step~1: writing~$\Pcalp v$ with respect to the Koornwinder basis and with explicit coefficients.}
We expand~$v$ with respect to the Koornwinder basis in~\eqref{Koornwinder-basis}:
\begin{equation} \label{2D-Koornwinder-expansion}
v(x,y) = \sum_{j+\ell=0}^{+\infty} \vbf_{j,\ell} \Phi_{j,\ell} (x,y).
\end{equation}
{In what follows, the indices appearing
under the summation symbol are always taken nonnegative;
moreover, we set~$\vbf_{j,\ell}$ equal to~$0$
if~$j$ or~$\ell$ are negative.}

Since~$\Pcalp v$ belongs to~$\Pbb_\p(\That)$, we expand it with respect to the Koornwinder basis:
\[
\Pcalp v(x,y)
= \sum_{j+\ell=0}^\p \vbfhat_{j,\ell} \Phi_{j,\ell}(x,y).
\]
By the orthogonality property~\eqref{definition:Pcalp1}, we readily get
\[
\Pcalp v(x,y)
= \sum_{j+\ell=0}^{\p-1} \vbf_{j,\ell} \Phi_{j,\ell}(x,y)
   + \sum_{j+\ell=\p} \vbfhat_{j,\ell} \Phi_{j,\ell}(x,y).
\]
The coefficients corresponding to the Koornwinder's polynomials of order~$\p$
are found using the facet orthogonality condition~\eqref{definition:Pcalp2}
on the facet~$\FplusThat$.
Notably, we impose the identity
\[
\int_{\FplusThat}
\Big[ {\sum_{j+\ell = \p+1}^{\infty}} \vbf_{j,\ell} \Phi_{j,\ell}(x,-1)
       + \sum_{j+\ell = \p} (\vbf_{j,\ell} - \vbfhat_{j,\ell})\Phi_{j,\ell}(x,-1)
\Big]
L_k(x) \ dx = 0
\qquad \forall k=0,\dots,\p.
\]
Using~\eqref{Koornwinder-basis}
and the fact that $\J_\ell^{2i+1}(-1) = (-1)^\ell$,
this identity can be rewritten as
\[
\begin{split}
& \sum_{j+\ell=\p}    \int_{\FplusThat} (\vbf_{j,\ell} - \vbfhat_{j,\ell}) (-1)^{\p-j} L_j(x) L_k(x) dx \\
& \qquad\qquad + \sum_{\ptilde=\p+1}^{+\infty} \sum_{j+\ell=\ptilde}
            \int_{\FplusThat} \vbf_{j,\ell} (-1)^{\ptilde-j} L_j(x) L_k(x) dx = 0
            \quad \forall  k=0,\dots,\p.
\end{split}
\]
Further using~\eqref{norm-Legendre} and the orthogonality property~\eqref{orthogonality-Legendre} of the Legendre polynomials,
we end up with
\[
(-1)^{\p-k} \frac{2}{2k+1}  [\vbf_{k,\p-k} - \vbfhat_{k,\p-k} ]
+ \sum_{\ptilde = \p+1}^{+\infty} (-1)^{\ptilde-k} \frac{2}{2k+1} \vbf_{k,\ptilde-k} = 0
\qquad\qquad \forall k=0,\dots,\p.
\]
In other words, the order~$\p$ coefficients of~$\Pcalp v$ are given by
\[
(-1)^{\p-k}\vbfhat_{k,\p-k}
= \sum_{\ptilde = \p}^{+\infty} (-1)^{\ptilde-k} \vbf_{k,\ptilde-k}
\qquad\qquad \forall k=0,\dots,\p.
\]
This is the 2D counterpart of~\eqref{expansion:Pcalp-1D}.
For completeness, we collect the above identities and arrive at
\[
\Pcalp v(x,y)
= \sum_{j+\ell=0}^{\p-1} \vbf_{j,\ell} \Phi_{j,\ell}(x,y)
   + \sum_{j+\ell=\p}
   \Big( \sum_{\ptilde = \p}^{+\infty} (-1)^{\ptilde-\p} \vbf_{j,\ptilde-j} \Big)
   \Phi_{j,\ell}(x,y).
\]
As in the 1D case, we deduce
\begin{equation} \label{expansion:v-Pcalpv:2D}
\begin{split}
(v-\Pcalp v)(x,y)
& = \sum_{j+\ell=\p}^{+\infty} \vbf_{j,\p-j} \Phi_{j,\ell}(x,y)
    - \sum_{j+\ell=\p} (\sum_{\ptilde=\p}^{+\infty} (-1)^{\ptilde-\p} \vbf_{j,\ptilde-j}) \Phi_{j,\ell}(x,y) \\
& = \sum_{j+\ell=\p+1}^{+\infty} \vbf_{j,\p-j} \Phi_{j,\ell}(x,y)
    - \sum_{j+\ell=\p} (\sum_{\ptilde=\p+1}^{+\infty} (-1)^{\ptilde-\p} \vbf_{j,\ptilde-j}) \Phi_{j,\ell}(x,y) \\
& = (v-\Pizp v)(x,y)
    - \sum_{j+\ell=\p} (\sum_{\ptilde=\p+1}^{+\infty} (-1)^{\ptilde-\p} \vbf_{j,\ptilde-j}) \Phi_{j,\ell}(x,y) .
\end{split}
\end{equation}

\paragraph*{Step~2: relate the tail of the Koornwinder coefficients to the trace of~$v-\Pizp v$ on the outflow facet.}
Using that~$J_\ell^{2j+1}(-1)=(-1)^\ell$ for all~$\ell$ and~$j$ in~$\Nbb$,
we write explicitly the trace of~$v-\Pizp v$ on~$\FplusThat$:
\[
(v-\Pizp v)(x,y)
= \sum_{j+\ell=\p+1}^{+\infty} \vbf_{j,\ell} \Phi_{j,\ell} (x,y),
\qquad
(v-\Pizp v)(x,-1)
= \sum_{j+\ell=\p+1}^{+\infty} (-1)^\ell \vbf_{j,\ell} L_j(x).
\]
The orthogonality property~\eqref{orthogonality-Legendre} of the Legendre polynomials over~$\Ihat := [-1,1]$ and~\eqref{norm-Legendre},
we write
\[
\begin{split}
\Norm{v-\Pizp v}_{0,\FplusThat}^2
& = \sum_{j=0}^{+\infty} \Norm{L_j}_{0,\Ihat}^2
        \Big( \sum_{\ell=\max(\p+1-j,0)}^{+\infty} (-1)^\ell \vbf_{j,\ell}   \Big)^2 \\
& = \sum_{j=0}^{+\infty} \frac{2}{2j+1}
        \Big( \sum_{\ell=\max(\p+1-j,0)}^{+\infty} (-1)^\ell \vbf_{j,\ell}   \Big)^2 .
\end{split}
\]
Next, we focus on the second term on the right-hand side of~\eqref{expansion:v-Pcalpv:2D}.
Using~\eqref{orthogonality:Koornwinder:2D},
we can write
\small{\begin{equation} \label{important-estimate-2D}
\begin{split}
& \Norm{\sum_{j+\ell=\p}
   \big(  \sum_{\ptilde=\p+1}^{+\infty} (-1)^{\ptilde-\p} \vbf_{j,\ptilde-j}  \big) \Phi_{j,\ell}}_{0,\That}^2
  = \sum_{j+\ell=\p} \big(  \sum_{\ptilde=\p+1}^{+\infty} (-1)^{\ptilde-\p} \vbf_{j,\ptilde-j}  \big)^2 \Norm{\Phi_{j,\ell}}_{0,\That}^2 \\
& =\sum_{j+\ell=\p} \big(  \sum_{\ptilde=\p+1}^{+\infty} (-1)^{\ptilde-\p} \vbf_{j,\ptilde-j}  \big)^2
  \frac{2}{2j+1}  \frac{1}{j+\ell+1}
= \frac{1}{\p+1} \sum_{j=0}^\p  \frac{2}{2j+1} \big(  \sum_{\ptilde=\p+1}^{+\infty} (-1)^{\ptilde-\p} \vbf_{j,\ptilde-j}  \big)^2  \\
& \le \frac{1}{\p+1} \sum_{j=0}^{+\infty} \frac{2}{2j+1}
                \big(  \sum_{\ptilde=\p+1}^{+\infty} (-1)^{\ptilde-\p} \vbf_{j,\ptilde-j}  \big)^2
  = \frac{1}{\p+1} \Norm{v-\Pizp v}_{0,\FplusThat}^2.
\end{split}
\end{equation}}\normalsize{}
The last identity follows from the fact that the quantity
in the parenthesis is squared
and the only relevant fact is the alternating sign of the
$(-1)^{\ptilde-\p}$ term.

\paragraph*{Step~3: deduce the desired bound using Lemma~\ref{lemma:Chernov-Melenk}.}
Combining the above inequality with~\eqref{expansion:v-Pcalpv:2D}
and using the 2D version of Lemma~\ref{lemma:Chernov-Melenk},
we deduce
\[
\Norm{v-\Pcalp v}_{0,\That}
\le \Norm{v-\Pizp}_{0,\That}
     + (\p+1)^{-\frac12} \Norm{v-\Pizp v}_{0,\FplusThat}
\lesssim \Norm{v-\Pizp}_{0,\That}
     + (\p+1)^{-1} \SemiNorm{v}_{1,\That}.
\]
{The assertion follows
noting that~$\Pcalp$ and~$\Pizp$ preserve polynomials of maximum degree~$\p$
and standard polynomial approximation results.}

\subsection{The 3D case} \label{subsection:3D-approx}
We prove Theorem~\ref{theorem:hp-approximation-CDG} in three dimensions.
Compared to the 1D and 2D cases of Sections~\ref{subsection:1D-approx} and~\ref{subsection:2D-approx},
the role of orthogonal basis functions is now played by tensor product Jacobi polynomials (of different orders)
collapsed on the tetrahedron using the 3D Duffy transformation;
in other words, we use 3D Koornwinder polynomials
following the construction in~\cite{Chernov:2012, Warburton-Hestaven:2003, Karniadakis-Sherwin:2005}.

Introduce the reference cube~$\Qhat = [-1,1]^3$
and the reference tetrahedron~$\That$ of vertices
\[
(-1,-1,-1); \qquad\qquad
(1,-1,-1); \qquad\qquad
(0,1,-1); \qquad\qquad
(0,0,1).
\]
We denote the lower face of~$\That$ by~$\FplusThat$;
it coincides with the reference triangle in Section~\ref{subsection:2D-approx}
at~$z=1$.

We consider the 3D Duffy transformation~$\Dcal:\Qhat \to \That$
that maps each~$(z_1,z_2,z_3)=\zbf$ in~$\Qhat$
into~$(x_1,x_2,x_3)=\xbf$ in~$\That$ as follows:
\[
x_1 = z_1 \frac{1-z_2}{2} \frac{1-z_3}{2},
\qquad\qquad
x_2 = z_2 \frac{1-z_3}{2},
\qquad\qquad
x_3 = z_3.
\]
The Jacobian  of the transformation is
\[
\det
\begin{bmatrix}
\frac{1-z_2}{2} \frac{1-z_3}{2} & 0                 & 0 \\
-\frac12 z_1 \frac{1-z_3}{2}    & \frac{1-z_3}{2}   & 0 \\
-\frac12 z_1 \frac{1-z_2}{2}    & -\frac{z_2}{2}    & 1 \\
\end{bmatrix}
= \frac{1-z_2}{2} \left( \frac{1-z_3}{2} \right)^2.
\]
We also compute the formal inverse~$\Dcal^{-1}$ of the transformation~$\Dcal$:
\[
z_3 = x_3,
\qquad\qquad
z_2 = x_2 \frac{2}{1-z_3} =  \frac{2x_2}{1-x_3},
\qquad\qquad
z_1 = x_1 \frac{4}{(1-z_2)(1-z_3)} = \frac{4x_1}{1-2x_2-x_3}.
\]
We introduce the following $L^2$-orthogonal basis over~$\That$:
given a multi-index $\jbf=(j_1,j_2,j_3)$ in~$\Nbb^3$,
\[
\Phi_{\jbf}(\xbf)
= L_{j_1}\left( \frac{4x_1}{1-2x_2-x_3} \right)
  \J_{j_2}^{2j_1+1} \left( \frac{2x_2}{1-x_3} \right) \left( \frac{1-2x_2-x_3}{2(1-x_3)}  \right)^{j_1}
  \J_{j_3}^{2j_1+2j_2+2}  (x_3) \left( \frac{1-x_3}{2} \right)^{j_1+j_2} .
\]
We only check the norm of the above Koornwinder polynomials.
Transforming~$\That$ into~$\Qhat$ by~$\Dcal^{-1}$
and recalling~\eqref{orthogonality-Jacobi},
we deduce
\begin{equation} \label{Koornwinder-norm-3D}
\begin{split}
& \int_{\That} \Phi_{\jbf}(\xbf) \Phi_{\jbf}(\xbf) d\xbf \\
& = \int_{\Qhat} L_{j_1}(z_1)^2
                \J_{j_2}^{2j_1+1} (z_2)^2 \left( \frac{1-z_2}{2} \right)^{2j_2+1}
                \J_{j_3}^{2j_1+2j_2+2}(z_3)^2 \left( \frac{1-z_3}{2} \right)^{2j_1+2j_2+2} \!\!\! d\zbf \\
& = \int_{\Ihat} L_{j_1}(z_1)^2 \ dz_1
  \int_{\Ihat}  \J_{j_2}^{2j_1+1} (z_2)^2 \left( \frac{1-z_2}{2} \right)^{2j_2+1} \!\!\! dz_2
  \int_{\Ihat}  \J_{j_3}^{2j_1+2j_2+2}(z_3)^2 \left( \frac{1-z_3}{2} \right)^{2j_1+2j_2+2} \!\!\! dz_3 \\
& = \frac{2}{2j_1+1} \frac{2}{2j_1+2j_2+2} \frac{2}{{2j_1+2j_2+2j_3+3}}
\qquad
\forall \jbf=(j_1,j_2,j_3) \in \Nbb^3.
\end{split}
\end{equation}
Next, we prove $\p$-optimal approximation properties in the $L^2$ norm of the operator~$\Pcalp$ given as in~\eqref{definition:Pcalp}
with the role of outflow face played by the facet~$\FplusThat$.

\paragraph*{Step~1: writing~$\Pcalp v$ with respect to the 3D version Koornwinder basis and with explicit coefficients.}
Consider the following expansion of~$v$:
\[
v(\xbf)
= \sum_{\vert\jbf\vert=0}^{+\infty}\vbf_{\jbf} \Phi_{\jbf}(\xbf).
\]
{In what follows, the indices appearing
under the summation symbol are always taken nonnegative;
$\vert \jbf \vert$ stands for $j_1+j_2+j_3$;
$\vbf_{\jbf}$ is set to~$0$ if any of the indices~$j_1$,
$j_2$, and~$j_3$ is negative.}

Using~\eqref{definition:Pcalp1}, we realise that
\[
v(\xbf)
= \sum_{\vert\jbf\vert=0}^{\p-1}\vbf_{\jbf} \Phi_{\jbf}(\xbf)
  + \sum_{\vert\jbf\vert=\p} \vbfhat_{\jbf} \Phi_{\jbf}(\xbf) ,
\]
where the coefficients~$\vbfhat_{\jbf}$,
$\vert\jbf\vert=\p$ are to be determined using~\eqref{definition:Pcalp2}.

First, we write the trace of any polynomial basis function on~$\FplusThat$,
i.e., we impose the passage through~$x_3=-1$:
\[
\Phi_{\jbf}(\xbf){}_{|\FplusThat}
= L_{j_1}\left( \frac{2x_1}{1-x_2} \right)
  \J_{j_2}^{2j_1+1} ( x_2 ) \left( \frac{1-x_2}{2} \right)^{j_1} (-1)^{j_3}.
\]
Let~$\Psi_{\ell_1,\ell_2}$ be the 2D Koornwinder basis in~\eqref{Koornwinder-basis}
on the reference triangle~$\FplusThat$.
We impose~\eqref{definition:Pcalp2}:
\[
\begin{split}
& \sum_{\vert \jbf \vert=\p}
\int_{\FplusThat} [\vbf_{\jbf} - \vbfhat_{\jbf}]
    (-1)^{j_3} L_{j_1}\left( \frac{2x_1}{1-x_2}  \right)
              L_{j_2}(x_2) \left( \frac{1-x_2}{2} \right)^{j_1} \Psi_{\ell_1,\ell_2} (x_1,x_2) dx_1 dx_2 \\
& + \sum_{\ptilde = \p+1}^{+\infty}
     \sum_{\vert\jbf\vert=\ptilde}
\int_{\FplusThat} \vbf_{\jbf}
    (-1)^{j_3} L_{j_1}\left( \frac{2x_1}{1-x_2}  \right)
              L_{j_2}(x_2) \left( \frac{1-x_2}{2} \right)^{j_1} \Psi_{\ell_1,\ell_2} (x_1,x_2) dx_1 dx_2
   =0 \\
& \qquad\qquad
{\forall \ell_1,\ell_2 \text{ such that } \ell_1+\ell_2\le \p.}
\end{split}
\]
We rewrite the above relation as
\[
\begin{split}
& \sum_{\vert \jbf \vert=\p}
\int_{\FplusThat} [\vbf_{\jbf} - \vbfhat_{\jbf}]
    (-1)^{j_3} \Psi_{j_1,j_2}(x_1,x_2) \Psi_{\ell_1,\ell_2} (x_1,x_2) dx_1 dx_2 \\
& + \sum_{\ptilde = \p+1}^{+\infty}
     \sum_{\vert\jbf\vert=\ptilde}
\int_{\FplusThat} \vbf_{\jbf}
    (-1)^{j_3} \Psi_{j_1,j_2}(x_1,x_2) \Psi_{\ell_1,\ell_2} (x_1,x_2) dx_1 dx_2
   =0
    \quad {\forall \ell_1,\ell_2 \text{ such that } \ell_1+\ell_2\le \p.}
\end{split}
\]
Using the orthogonality property~\eqref{orthogonality:Koornwinder:2D}
of the Koornwinder polynomials in 2D,
we deduce
{\[
\begin{split}
& (-1)^{\p-\ell_1-\ell_2}
\frac{2}{2\ell_1+1}\frac{1}{\ell_1+\ell_2+1}
(\vbf_{\ell_1,\ell_2,\p-\ell_1-\ell_2}
-\vbfhat_{\ell_1,\ell_2,\p-\ell_1-\ell_2}) \\
& + \sum_{\ptilde=\p+1}^{+\infty}
  (-1)^{\ptilde-(\ell_1+\ell_2)}
\frac{2}{2\ell_1+1}\frac{1}{\ell_1+\ell_2+1}
\vbf_{\ell_1,\ell_2,\ptilde-\ell_1-\ell_2} =0
\qquad \forall \ell_1,\ell_2 \text{ such that } \ell_1+\ell_2\le \p.
\end{split}
\]}
In other words, we have
\[
\vbfhat_{\ell_1,\ell_2,\p-\ell_1-\ell_2}
= \sum_{\ptilde=\p}^{+\infty}
  (-1)^{\ptilde-\p}
\vbf_{\ell_1,\ell_2,\ptilde-\ell_1-\ell_2}
\qquad\qquad
{\forall \ell_1,\ell_2 \text{ such that } \ell_1+\ell_2\le \p.}
\]
This provides us with the representation
\[
\Pcalp v(\xbf)
= \sum_{\vert \jbf \vert=0}^{\p-1}
    \vbf_{\jbf} \Phi_{\jbf}(\xbf)
  + \sum_{\vert \jbf \vert=\p}
    \big( \sum_{\ptilde=\p}^{+\infty}
       (-1)^{\ptilde-\p}
     {\vbf_{j_1,j_2,\ptilde-j_1-j_2}}
    \big) \Phi_{\jbf}(\xbf).
\]
We get
\begin{equation} \label{expansion:v-Pcalpv:3D}
\begin{split}
& (\vbf-\Pcalp\vbf)(\xbf)
 = \sum_{\vert \jbf \vert=\p}^{+\infty}
    \vbf_{j_1,j_2,\p-j_1-j_2} \Phi_j(\xbf)
    - \sum_{\vert \jbf \vert=\p}
    \big( \sum_{\ptilde=\p}^{+\infty}
  (-1)^{\ptilde-\p}
\vbf_{{j_1,j_2,\ptilde-j_1-j_2}}
    \big) \Phi_{\jbf}(\xbf)\\
& = \sum_{\vert \jbf \vert=\p+1}^{+\infty}
    \vbf_{j_1,j_2,\p-j_1-j_2} \Phi_j(\xbf)
    - \sum_{\vert \jbf \vert=\p}
    \big( \sum_{\ptilde=\p+1}^{+\infty}
  (-1)^{\ptilde-\p}
\vbf_{{j_1,j_2,\ptilde-j_1-j_2}}
    \big) \Phi_{\jbf}(\xbf) \\
& = (\vbf-\Pizp \vbf)(\xbf)
    - \sum_{\vert \jbf \vert=\p}
    \big( \sum_{\ptilde=\p+1}^{+\infty}
  (-1)^{\ptilde-\p}
\vbf_{{j_1,j_2,\ptilde-j_1-j_2}}
    \big) \Phi_{\jbf}(\xbf) .
\end{split}
\end{equation}

\paragraph*{Step~2: relate the tail of the 3D version Koornwinder coefficients to the trace of~$v-\Pizp v$ on the outflow facet.}
The restriction of~$\vbf-\Pizp\vbf$ on the outflow facet~$\FplusThat$
reads
\[
\begin{split}
(\vbf - \Pcalp \vbf)(\xbf)_{|\FplusThat}
& = \sum_{\vert \jbf\vert=\p+1}^{+\infty}
   \vbf_{\jbf} L_{j_1} \left( \frac{2x_1}{1-x_2} \right)
    \J_{j_2}^{2j_1+1} (x_2)
    \left( \frac{1-x_2}{2} \right)^{j_1} (-1)^{j_3} \\
& = \sum_{\vert \jbf\vert=\p+1}^{+\infty}
   \vbf_{\jbf} \Psi_{j_1,j_2}(\xbf) (-1)^{j_3} .
\end{split}
\]
We take the $L^2$ norm on both sides over~$\FplusThat$,
use the orthogonality property~\eqref{orthogonality:Koornwinder:2D} of Koornwinder polynomials over~$\FplusThat$,
and get,
for~$\jtilde$ equal to $\max(\p+1-(j_1+j_2),0)$,
\small{\[
\Norm{\vbf - \Pcalp \vbf}_{0,\FplusThat}
= \sum_{j_1+j_2=0}^{+\infty}
    \Norm{\Psi_{j_1,j_2}}_{0,\FplusThat}^2
    \big( \sum_{j_3=\jtilde}
          (-1)^{j_3} \vbf_{\jbf}    \big)^2
 = \sum_{j_1+j_2=0}^{+\infty}
    \frac{2}{2j_2+1} \frac{1}{j_1+j_2+1}
    \big( \sum_{j_3=\jtilde}
          (-1)^{j_3} \vbf_{\jbf}    \big)^2 .
\]}\normalsize{}
We show a bound on the $L^2$ norm over~$\That$ of the second term on the right-hand side of~\eqref{expansion:v-Pcalpv:3D}
in terms of the $L^2$ norm of~$\vbf - \Pcalp \vbf$ on~$\FplusThat$.
To this aim, we use~\eqref{Koornwinder-norm-3D}:
\small{\[
\begin{split}
& \Norm{\sum_{\vert \jbf \vert=\p}
    \big( \sum_{\ptilde=\p+1}^{+\infty}
  (-1)^{\ptilde-\p}
\vbf_{{j_1,j_2,\ptilde-j_1-j_2}}
    \big) \Phi_{\jbf}(\xbf)}_{0,\That}^2
= \sum_{\vert \jbf \vert=\p}
    \big( \sum_{\ptilde=\p+1}^{+\infty}
        (-1)^{\ptilde-\p}
        \vbf_{{j_1,j_2,\ptilde-j_1-j_2}}  \big)^2
        \Norm{\Phi_{\jbf}}_{0,\That}^2 \\
& = \sum_{\vert \jbf \vert=\p}
    \big( \sum_{\ptilde=\p+1}^{+\infty}
        (-1)^{\ptilde-\p}
        \vbf_{{j_1,j_2,\ptilde-j_1-j_2}}  \big)^2
        \frac{2}{2j_1+1}\frac{1}{j_1+j_2+1}  \frac{1}{j_1+j_2+j_3+1}\\
& = \frac{1}{\p+1} \sum_{j_1+j_2=0}^{\p}
    \big( \sum_{\ptilde=\p+1}^{+\infty}
        (-1)^{\ptilde-\p}
        \vbf_{{j_1,j_2,\ptilde-j_1-j_2}}  \big)^2
        \frac{2}{2j_1+1}\frac{1}{j_1+j_2+1} \\
& \le \frac{1}{\p+1} \sum_{j_1+j_2=0}^{+\infty}
    \big( \sum_{\ptilde=\p+1}^{+\infty}
        (-1)^{\ptilde-\p}
        \vbf_{{j_1,j_2,\ptilde-j_1-j_2}}  \big)^2
        \frac{2}{2j_1+1}\frac{1}{j_1+j_2+1}
    = \frac{1}{\p+1} \Norm{\vbf-\Pizp\vbf}_{0,\FplusThat}^2.
\end{split}
\]}\normalsize{}

\paragraph*{Step~3: deduce the desired bound using optimal trace error estimates.}
Combining the above inequality with~\eqref{expansion:v-Pcalpv:2D}
and using the 3D version of Lemma~\ref{lemma:Chernov-Melenk},
we deduce
\[
\Norm{v-\Pcalp v}_{0,\That}
\le \Norm{v-\Pizp}_{0,\That}
     + (\p+1)^{-\frac12} \Norm{v-\Pizp v}_{0,\FplusThat}
\lesssim \Norm{v-\Pizp}_{0,\That}
     + (\p+1)^{-1} \SemiNorm{v}_{1,\That}.
\]
{The assertion follows
noting that~$\Pcalp$ and~$\Pizp$ preserve polynomials of maximum degree~$\p$
and standard polynomial approximation results.}

\subsection{Trace-type estimates for the CDG operator} \label{subsection:trace-estimates}
The CDG operator satisfies $\h\p$-optimal approximation properties in the $L^2$ norm; see Theorem~\ref{theorem:hp-approximation-CDG}.
Here, we investigate $\h\p$-optimal approximation properties in the $L^2$ norm on the boundary of a simplex~$\T$,
i.e., the counterpart of Lemma~\ref{lemma:Chernov-Melenk} for the projector~$\Pcalp$.

To {this} aim, we show separate bounds on the facets.
First, we consider the case of the (unique) outflow facet~$\F^+_\T$.
By~\eqref{definition:Pcalp2} and Lemma~\ref{lemma:Chernov-Melenk}, we get
\[
\Norm{v-\Pcalp v}_{0,\F^+_\T}
\le \Norm{v-\Pizp v}_{0,\F^+_\T}
\lesssim \left(\frac{\hT}{\p}\right)^{\frac12} \SemiNorm{v}_{1,\T}
\qquad\qquad\qquad \forall v \in H^1(\T).
\]
Approximation properties on the inflow facets are more elaborated
and are discussed in the next result.
\begin{proposition} \label{proposition:trace-estimates}
Let the assumptions (\textbf{A1}) and (\textbf{A2}) be valid,
$\T$ be in~$\taun$,
$\F^-_\T$ be any of the inflow facets of~$\T$,
and~$\Pizp$ be the operator in~\eqref{L2-projector}.
Then, there exist positive constants~$C_1$ and~$C_2$ independent of~$\h$ and~$\p$ such that,
{for every $v$ in $H^{k+1}(\T)$, $k$ positive,}
the following estimate holds true:
\begin{equation} \label{hp-trace-estimates}
\Norm{v-\Pcalp v}_{0,\F^-_\T}
\le C_1 \Norm{v-\Pizp v}_{0,\partial \T}
\le C_2 \frac{\h^{\min(\p,k)+\frac12}}{\p^{k+\frac12}} \Norm{v}_{k+1,\T} .
\end{equation}
\end{proposition}
\begin{proof}
The proof of the 1D case is fairly simple.
The two and three dimensional cases are more elaborated;
we provide details for the 2D case, as the 3D case can be dealt with similarly.
Throughout, we employ the notation of Sections
\ref{subsection:CDG-projector}--\ref{subsection:3D-approx}.
We prove the assertion on the reference simplex~$\That$;
the general statement follows from a scaling argument.

\paragraph*{Proof of~\eqref{hp-trace-estimates} in 1D.}
Here~$\That=[-1,1]$.
From~\eqref{expansion:v-Pcalpv:1D},
\eqref{trace:v-Pizpv:1D},
and~$L_j(1)=1$ for all $j$ in~$\Nbb$,
we deduce
\[
\vert (v-\Pcalp v)(1) \vert
\le \vert (v-\Pizp v)(-1) \vert + \vert (v-\Pizp v)(1) \vert.
\]
The assertion follows from Lemma~\ref{lemma:Chernov-Melenk}.

\paragraph*{Proof of~\eqref{hp-trace-estimates} in 2D.}
Here, $\That$ is as in~\eqref{reference-triangle}.
{Recall expansion~\eqref{2D-Koornwinder-expansion}
with the caveat that coefficients
related to negative indices are set to zero.}

Introduce
\[
\REM:=
\sum_{j+\ell=\p} (\sum_{\ptilde=\p+1}^{+\infty} (-1)^{\ptilde-\p} \vbf_{j,\ptilde-j}) \Phi_{j,\ell}(x,y)
\]
and
\[
{\widehat C_\p} := \max_{j+\ell=\p}  \Norm{\Phi_{j,\ell}(x,y)}_{0,\FminusThat}.
\]
We have
\[
\Norm{\REM}_{0,\FminusThat}^2
\lesssim
\sum_{j+\ell=\p} \vert \sum_{\ptilde=\p+1}^{+\infty} (-1)^{\ptilde-\p} \vbf_{j,\ptilde-j}\vert^2
                \Norm{\Phi_{j,\ell}(x,y)}_{0,\FminusThat}^2
\le {\widehat C_\p}^2 \sum_{j+\ell=\p} \vert \sum_{\ptilde=\p+1}^{+\infty} (-1)^{\ptilde-\p} \vbf_{j,\ptilde-j}\vert^2,
\]
whence we deduce
\[
\begin{split}
\Norm{\REM}_{0,\FminusThat}^2
& \lesssim {\widehat C_\p}^2 \sum_{j+\ell=\p} \vert \sum_{\ptilde=\p+1}^{+\infty} (-1)^{\ptilde-\p} \vbf_{j,\ptilde-j}\vert^2
            \frac{2}{2j+1} \frac{2j+1}{2}\\
& \le {\widehat C_\p}^2 \frac{2\p+1}{2} \sum_{j+\ell=\p} \vert \sum_{\ptilde=\p+1}^{+\infty} (-1)^{\ptilde-\p} \vbf_{j,\ptilde-j}\vert^2 \frac{2}{2j+1} .
\end{split}
\]
We proceed as in~\eqref{important-estimate-2D} and arrive at
\[
\Norm{\REM}_{0,\FminusThat}^2
\lesssim {\widehat C_\p}^2 \frac{2\p+1}{2}
        \sum_{j+\ell=\p} \vert \sum_{\ptilde=\p+1}^{+\infty} (-1)^{\ptilde-\p} \vbf_{j,\ptilde-j}\vert^2
        \frac{2}{2j+1}
 \le {\widehat C_\p}^2 \frac{2\p+1}{2} \Norm{v-\Pizp v}_{0,\FplusThat}^2.
\]
Combining~\eqref{expansion:v-Pcalpv:2D} with the inequality above yields
\begin{equation} \label{bound-outflow-with-inflow}
\Norm{v-\Pcalp v}_{0,\FminusThat}
\le \Norm{v-\Pizp v}_{0,\FminusThat} + \Norm{\REM}_{0,\FminusThat}
\lesssim \Norm{v-\Pizp v}_{0,\FminusThat} + {\widehat C_\p} \p^\frac12 \Norm{v-\Pizp v}_{0,\FplusThat}.
\end{equation}
If we were able to prove that~${\widehat C_\p} \lesssim \p^{-\frac12}$,
then the assertion would follow from Lemma~\ref{lemma:Chernov-Melenk} and~\eqref{bound-outflow-with-inflow}.
In fact, we would write
\[
\Norm{v-\Pcalp v}_{0,\FminusThat}
\le \Norm{v-\Pizp v}_{0,\FminusThat} + \Norm{\REM}_{0,\FminusThat}
\lesssim  \Norm{v-\Pizp v}_{0,\FminusThat} + \Norm{v-\Pizp v}_{0,\FplusThat}.
\]
We are left with showing~${\widehat C_\p} \lesssim \p^{-\frac12}$.
For the sake of presentation, we fix $\FminusThat$ to be the inflow facet of~$\That$ with negative~$x$ coordinates.
If this is the case, then a direct application of definition~\eqref{Koornwinder-basis} gives
\[
\Phi_{j,\ell}{}_{|\FminusThat}
= (-1)^j \J_\ell^{2j+1}(y) \left( \frac{1-y}{2} \right)^j.
\]
Taking into account the Jacobian of the transformation mapping the facet~$\FminusThat$
into the ``vertical'' facet $\{-1\} \times [-1,1]$
and recalling~\eqref{orthogonality-Jacobi},
we deduce
\[
\Norm{\Phi_{j,\ell}}_{0,\FminusThat}^2
= \int_{-1}^1 \left[\J_\ell^{2j+1}(y)\right]^2
                \left( \frac{1-y}{2} \right)^{2j+1}
= \frac{2^{2j+2}}{2j+1+2\ell+1} \frac{1}{2^{2j+1}}
= \frac{1}{1+j+\ell}.
\]
Since we are considering couples $(j,\ell)$ such that~$j+\ell=\p$, we arrive at
\[
\Norm{\Phi_{j,\ell}}_{0,\FminusThat}^2 = \frac{1}{\p+1},
\]
i.e., ${\widehat C_\p} \lesssim \p^{-\frac12}$.
\end{proof}

\section{A priori analysis} \label{section:convergence}
We prove $\h\p$-optimal a priori error estimates for method~\eqref{method}.
The proof is essentially the $\h\p$-version of that of~\cite[Theorem~$2.2$]{Cockburn-Dong-Guzman:2008}.
We report here the details for three reasons:
for the sake of completeness;
since we are interested in $\h\p$-optimal a priori bounds;
because we have different assumptions,
namely we require~\eqref{standard-assumption-convection} for well posedness of the method,
but we have no restrictions on the mesh size~$\h$
as for instance required in~\cite{Johnson-Pitkaranta:1986}.

\begin{theorem} \label{theorem:hp-a-priori}
Let~$\taun$ be a shape-regular, quasi uniform mesh satisfying the assumptions (\textbf{A1})--(\textbf{A2}),
and~$u$ and~$\uh$ be the solutions to~\eqref{weak-formulation} and~\eqref{method}.
If~\eqref{standard-assumption-convection} holds true
and~$u$ belongs to~$H^{k+1}(\Omega)$,
{$k$ positive,}
then there exists a positive constant~$C$ independent of~$\h$ and~$\p$
but possibly depending on~$\sigma$ such that
\[
\Norm{u-\uh}_{0,\Omega}
\le C \frac{\h^{\min(k,\p)+1}}{\p^{k+1}} \Norm{u}_{k+1,\Omega}.
\]
\end{theorem}
\begin{proof}
Introduce
\[
\eh:=\uh - \Pcalp u.
\]
The rewriting of the bilinear form in~\eqref{Cockburn-formulation}
and the Galerkin orthogonality~\eqref{Galerkin-orthogonality}
imply
\[
\Bcal(\eh,\vh)
= \Bcal(u-\Pcalp u, \vh)
= \sum_{j=1}^3 T_j,
\]
where
\[
\begin{split}
& T_1 := -(u-\Pcalp u, \betabold \cdot \nablah \vh)_{0,\Omega} , \\
& T_2 := \sum_{\T \in \taun} ( (u-\Pcalp u)^-, \nbfT \cdot \betabold \ \jump{\vh})_{0,\GammaTin}
+ \sum_{\T\in\taun}( u-\Pcalp u, \nbfT \cdot \betabold \ \vh)_{0,\partial\T\cap\Gammaplus} , \\
& T_3 := (c (u-\Pcalp u), \vh)_{0,\Omega} .
\end{split}
\]
The term~$T_1$ vanishes due to~\eqref{definition:Pcalp1}.

As for the term~$T_2$, some manipulations imply
\[
T_2
= -\sum_{\T \in \taun} ( u- \Pcalp u, \nbfT \cdot \betabold \jump{\vh})_{0,\GammaTout}
  + \sum_{\T\in\taun}( u-\Pcalp u, \nbfT \cdot \betabold \ \vh)_{0,\partial\T\cap\Gammaplus} .
\]
The last identity is a consequence of the definition of outflow facets.
Recalling~\eqref{definition:Pcalp2}, also~$T_2$ vanishes.
We end up with
\[
\Bcal(\eh,\vh)
= (c (u-\Pcalp u), \vh)_{0,\Omega} .
\]
This means that $\eh$ solves method~\eqref{method}
with data~$f$ and~$g$ given by~$c(u-\Pcalp u)$ and~$0$, respectively.

This identity and the stability estimate in Proposition~\ref{proposition:stability-method} imply
\[
\Norm{\eh}_{0,\Omega} \lesssim \Norm{c(u-\Pcalp u)}_{0,\Omega}.
\]
{The triangle inequality and the above bound entail
\[
\Norm{u-\uh}_{0,\Omega}
\le \Norm{u-\Pcalp u}_{0,\Omega} + \Norm{u-\uh}_{0,\Omega}
\lesssim
\Norm{u-\Pcalp u}_{0,\Omega} + \Norm{c(u-\Pcalp u)}_{0,\Omega}.
\]
The assertion follows using Theorem~\ref{theorem:hp-approximation-CDG}.}
\end{proof}

We also prove $\h\p$-optimal convergence for method~\eqref{method}
in the full DG norm
\small{\[
|||\vh|||_{\DG}^2
:=\!\! \sum_{\T \in \taun}
\left( \Norm{c \ \vh}_{0,\T}^2
    + \Norm{\vert\nbfT \cdot \betabold \vert^{\frac12}\vh^-}_{0,\GammaTin \cap \Gammaminus}^2
    \!\!\!\!+ \Norm{\vert\nbfT \cdot \betabold \vert^{\frac12}\vh^+}_{0,\GammaTout \cap \Gammaplus}^2
    \!\!\!\!+ \Norm{\vert\nbfT \cdot \betabold \vert^{\frac12}\jump{\vh}}_{0,\GammaTin \setminus \Gammaminus}^2 \right).
\]}\normalsize{}
\begin{theorem} \label{theorem:convergence-full-norm}
Let~$\taun$ be a shape-regular, quasi uniform mesh satisfying the assumptions (\textbf{A1})--(\textbf{A2}),
and~$u$ and~$\uh$ be the solutions to~\eqref{weak-formulation} and~\eqref{method}.
If~\eqref{standard-assumption-convection} holds true
and~$u$ belongs to~$H^{k+1}(\Omega)$, {$k$ positive,}
then there exists a positive constant~$C$ independent of~$\h$ and~$\p$
but possibly depending on~$\sigma$ such that
\[
|||u-\uh|||_{\DG}
\le C \frac{\h^{\min(k,\p)+\frac12}}{\p^{k+\frac12}} \Norm{u}_{k+1,\Omega}.
\]
\end{theorem}
\begin{proof}
This is the simplicial version of \cite[Theorem~3.7]{Houston-Schwab-Suli:2000}.
As in the proof of that result, given~$\eta:=u-\Pizp u$,
it is possible to show that
\small{\[
\begin{split}
& ||| u-\uh|||_{\DG} \\
& \lesssim (\sum_{\T \in \taun} \Norm{\eta}_{0,\T}^2)^\frac12
    + \Big(\sum_{\T \in \taun}
        \Big(\Norm{\vert\nbfT \cdot \betabold \vert^{\frac12}\eta^+}_{0,\GammaTout \cap \Gammaplus}^2
         + \Norm{\vert\nbfT \cdot \betabold \vert^{\frac12}\eta^+}_{0,\GammaTin \cap \Gammaminus}^2\\
& \qquad\qquad  + \Norm{\vert\nbfT \cdot \betabold \vert^{\frac12}\eta^-}_{0,\GammaTout \setminus \Gammaminus}^2
         + \Norm{\vert\nbfT \cdot \betabold \vert^{\frac12}\eta^+}_{0,\GammaTout \setminus \Gammaminus}^2\Big)\Big)^\frac12
 =: A+B.
\end{split}
\]}\normalsize{}
The hidden constant above is independent of~$\h$ and~$\p$.

The term~$A$ is dealt with using standard
$\h\p$ {Babu\v ska-Suri type} approximation properties
of the $L^2$ projection {in the $L^2$ norm};
see, e.g., \cite[Lemma~4.5]{Babuvska-Suri:1987}.
The term~$B$ is dealt with using
{approximation properties
of the $L^2$ projection in the boundary $L^2$ norm
as discussed in Lemma~\ref{lemma:Chernov-Melenk}.}
\end{proof}

\noindent Theorems~\ref{theorem:hp-a-priori} and~\ref{theorem:convergence-full-norm}
improve the current state of the art of the literature along different directions:
\begin{itemize}
    \item \cite[Theorem~2.2]{Cockburn-Dong-Guzman:2008} shows $\h$-optimal convergence in the $L^2$ norm;
    here, we provide full explicit track of the $\p$-convergence;
    \item \cite[Theorem~3.7]{Houston-Schwab-Suli:2000}
    and \cite[Theorem~5.1]{Houston-Schwab-Suli:2002}
    display $\h\p$-optimal convergence in the full DG norm on Cartesian-type meshes;
    here, we show optimal $\h\p$-optimal convergence in the full DG norm on special simplicial meshes
    and also improved convergence (on simplicial meshes) in the $L^2$ norm.
\end{itemize}
The theoretical limitations of Theorems~\ref{theorem:hp-a-priori} and~\ref{theorem:convergence-full-norm} are the use
\begin{itemize}
    \item of the special meshes in Section~\ref{subsection:meshes} from~\cite{Cockburn-Dong-Guzman:2008};
    \item of assumption~\eqref{standard-assumption-convection}.
\end{itemize}
In Section~\ref{section:numerical-experiments} below,
we investigate whether the two limitations above can be overcome in practice
or are necessary condition for the $\h\p$-optimality.

\begin{remark}
Following the proof of Theorem~\ref{theorem:convergence-full-norm},
it is possible to derive $\h\p$-optimal convergence
of DG methods on simplicial meshes
for diffusion-advection-reaction problems
as in~\cite{Houston-Schwab-Suli:2002}.
In words, one needs to care also of an interior penalty discontinuous discretization
of a diffusion term;
the error of the method is then estimated from above by the distance in a DG type norm
between the exact solution and its $L^2$ projection.
Such a quantity is then estimated optimally using
the {Babu\v ska-Suri type} approximation properties
of the $L^2$ projection in the $L^2$ norm (using~\cite{Babuvska-Suri:1987}),
in the $H^1$ seminorm (using~\cite{Canuto-Quarteroni:1982}),
and in the $L^2$ norm on the boundary (using Lemma~\ref{lemma:Chernov-Melenk}).
\end{remark}

\section{Numerical experiments} \label{section:numerical-experiments}
We assess the convergence rates
of method~\eqref{method} in the $L^2$ and full $\DG$ norms
proven in Theorems~\ref{theorem:hp-a-priori} and~\ref{theorem:convergence-full-norm}.
We focus on the $\p$-version of the method in two dimensions,
fix two triangular meshes of~$50$ and~$32$ triangles
partitioning the square domain~$\Omega:=(-1,1)^2$,
see Figure~\ref{figure:meshes},
and use polynomial degrees~$\p$ from~$1$ to~$40$.
Standard composite Gaussian quadrature is employed throughout.
We pick an $L^2$ orthonormal basis over simplices as that constructed in~\cite{Karniadakis-Sherwin:2005}.

We consider two test cases.
The first one involves a constant convection field such that
assumption~\eqref{standard-assumption-convection} holds true
and  the meshes satisfy the construction in Section~\ref{subsection:meshes}.
Instead, a second test case is devoted to check the performance of the scheme
even if the convection field and the meshes do not satisfy the assumptions
needed in the analysis of Sections~\ref{section:main} and~\ref{section:convergence}.

\begin{figure}[h]
\centering
 \includegraphics[height=2in,width=2in]{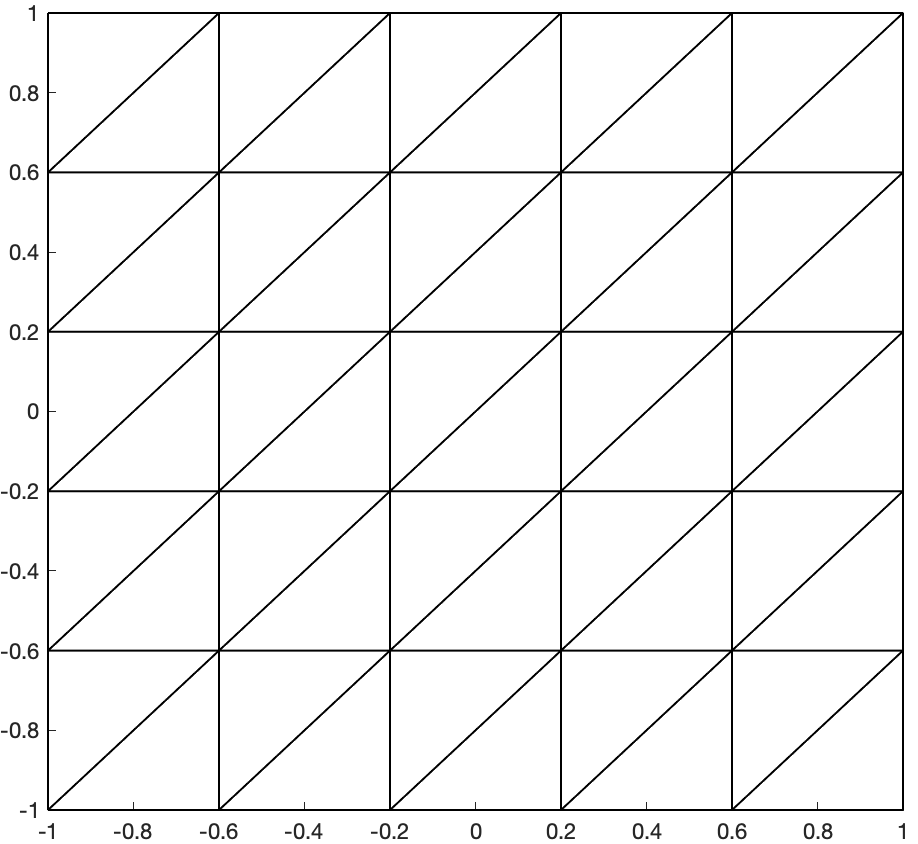}
 \qquad\qquad\qquad\qquad
 \includegraphics[height=2in,width=2in]{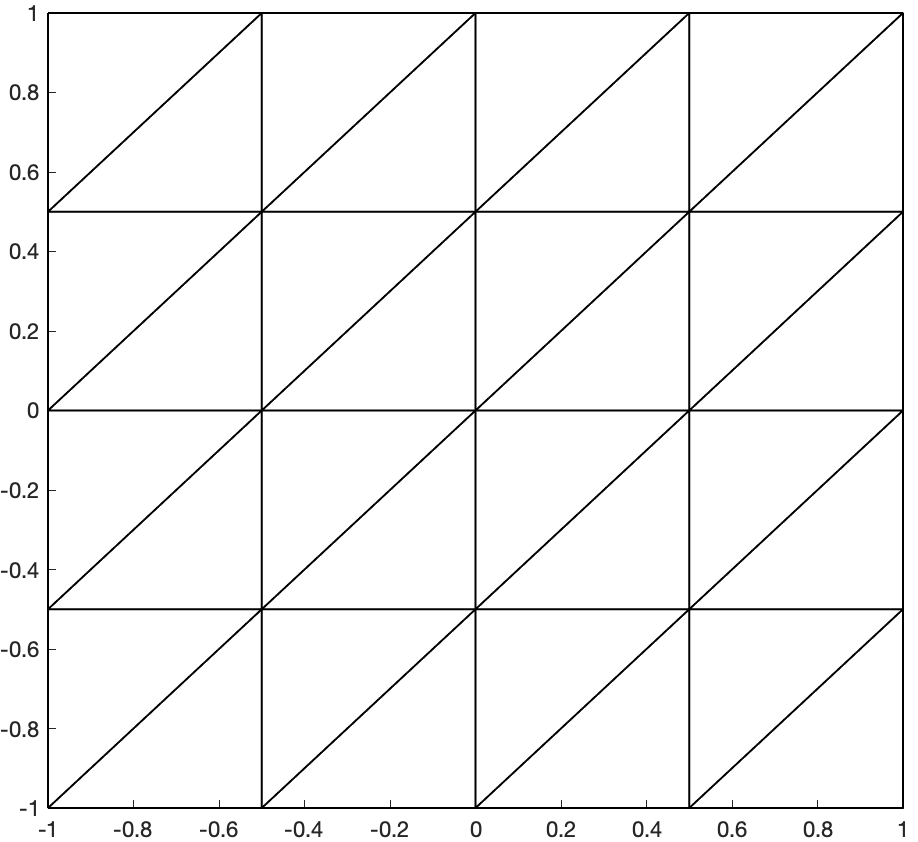}
\caption{
\emph{Left-panel}: mesh consisting of 50 triangles.
\emph{Right-panel}:  mesh consisting of 32 triangles.}
 \label{figure:meshes}
\end{figure}

\paragraph*{Test case~1.}
As in~\cite[Section~5.2]{Castillo-Cockburn-Perugia-Schoetzau:2000},
we consider the convection field~$\betabold =(1,1)$,
the reaction coefficient~$c=1$,
and the exact solution, for given positive~$\alpha$,
\begin{equation} \label{u1}
u(x,y)
= \begin{cases}
\cos (\pi y/2 )                 & \text{ in } (-1,0] \times (-1,1)  \\
\cos (\pi y/2 ) + x^\alpha      & \text{ in } (0 ,1) \times (-1,1). \\
\end{cases}
\end{equation}
The right-hand side~$f$ is computed accordingly.
The two meshes in Figure~\ref{figure:meshes}
are such that each element has only one outflow facet for the given convection field~$\betabold$.

The solution~$u$ belongs to $H^{1/2+\alpha -\epsilon}(\Omega)$
for any positive and arbitrarily small~$\epsilon$.
The singularity is located in the interior of a mesh cell
for the mesh in Figure~\ref{figure:meshes} (\emph{left-panel})
and at the interface of several cell elements
for the mesh in Figure~\ref{figure:meshes} (\emph{right-panel}).

We report the results in Figure~\ref{figure:u1-50} and~\ref{figure:u1-32}
for the~$50$ and~$32$ triangles cases, respectively.
We pick~$\alpha=0.5$, $1.5$, and~$2.5$.

\begin{figure}[h]
\centering
 \includegraphics[height=2.7in,width=2.7in]{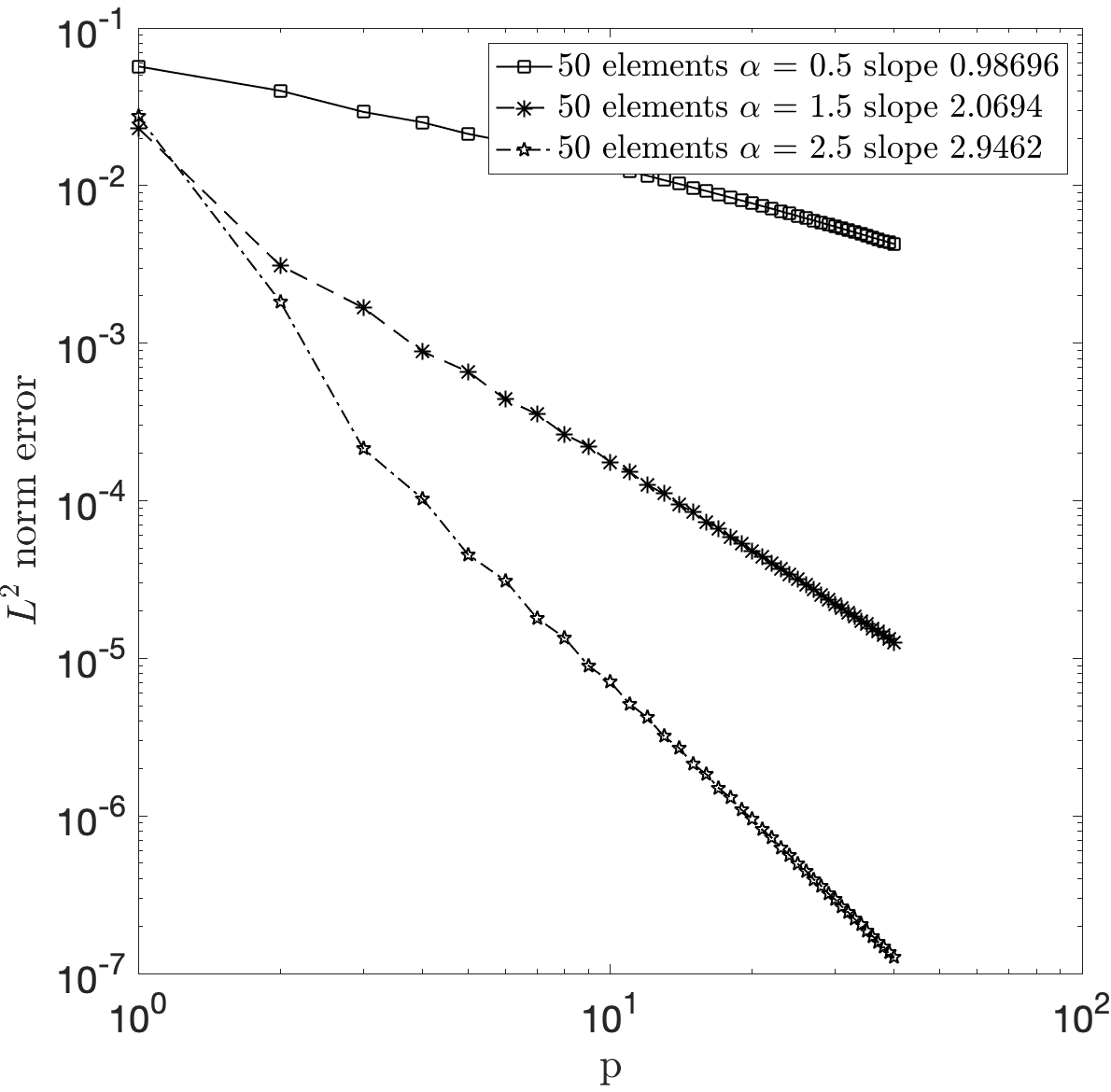}
 \qquad
 \includegraphics[height=2.7in,width=2.7in]{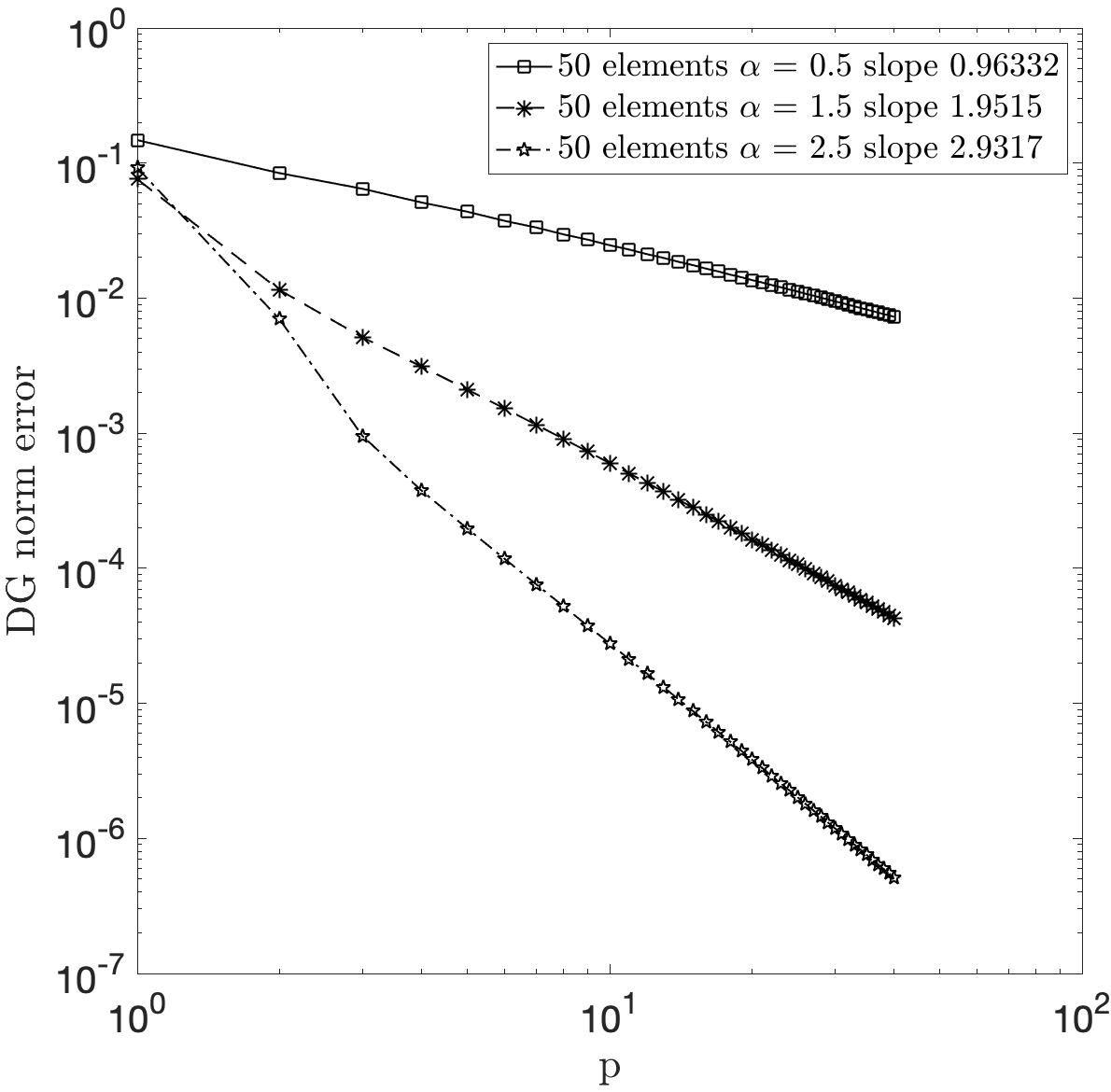}
\caption{$\p$-version for the case of constant convection field~$\betabold$.
We consider the exact solution~$u$ in~\eqref{u1}
and fix the triangular mesh with 50 elements as in Figure~\ref{figure:meshes}
so as the mesh assumptions in Section~\ref{subsection:meshes} are satisfied.
The singularity lies in the interior of a cell element.
\emph{Left-panel}: $L^2$ norm error.
\emph{Right-panel}: DG norm error.}
\label{figure:u1-50}
\end{figure}
\begin{figure}[h]
\centering
 \includegraphics[height=2.7in,width=2.7in]{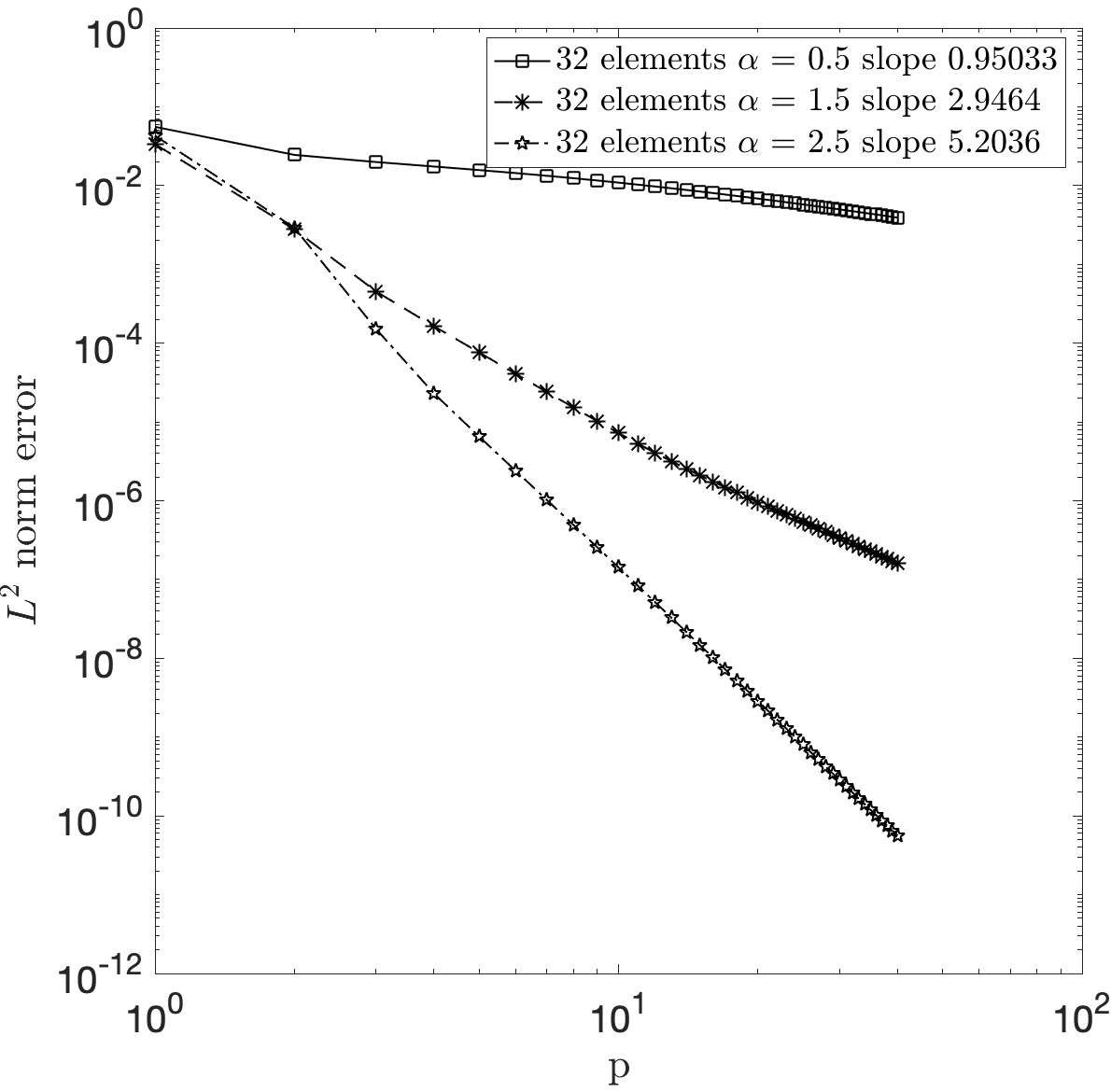}
 \qquad
 \includegraphics[height=2.7in,width=2.7in]{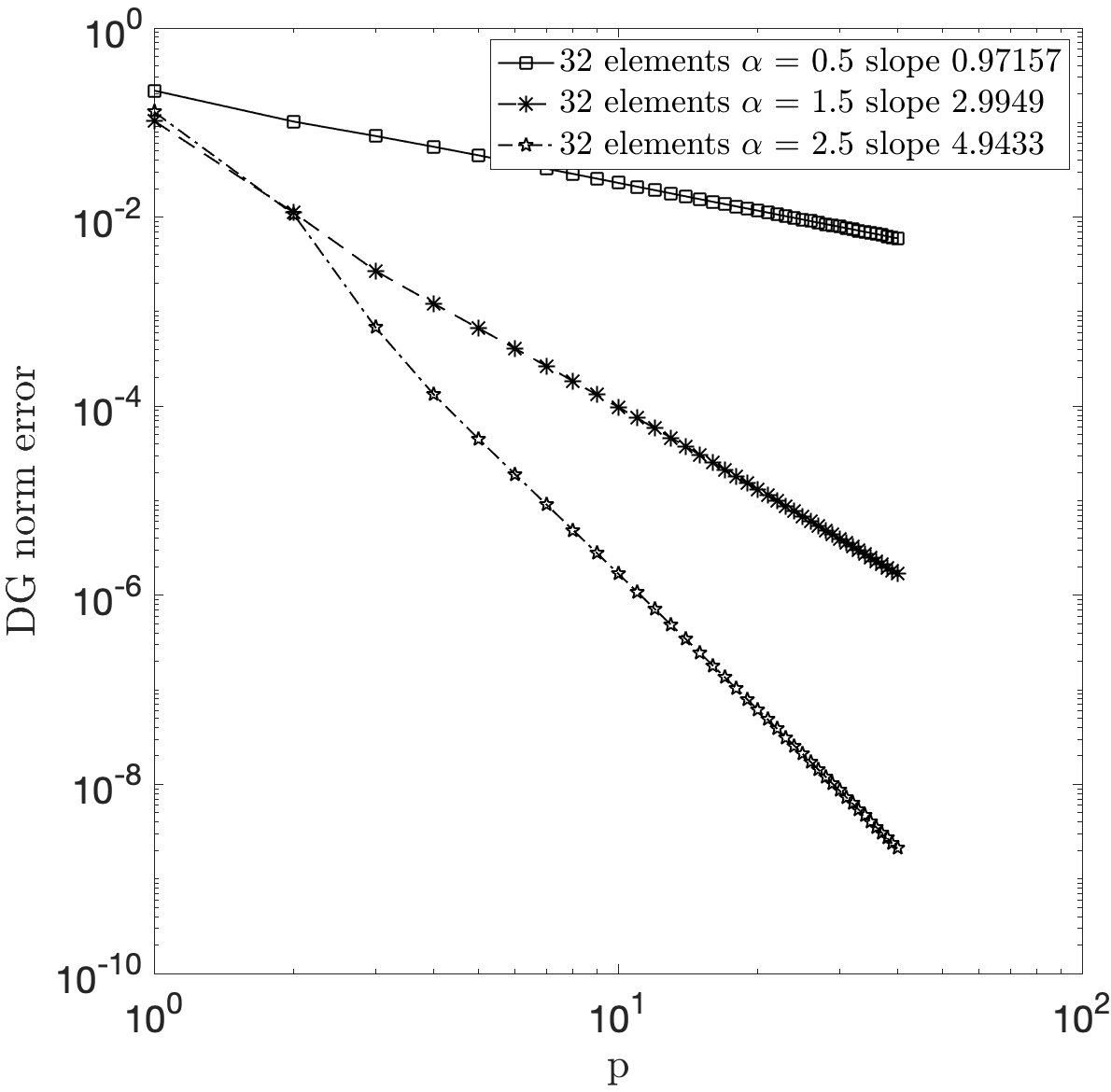}
\caption{$\p$-version for the case of constant convection field~$\betabold$.
We consider the exact solution~$u$ in~\eqref{u1}
and fix the triangular mesh with 32 elements as in Figure~\ref{figure:meshes}
so as the mesh assumptions in Section~\ref{subsection:meshes} are satisfied.
The singularity lies at the interface of several cell elements.
\emph{Left-panel}: $L^2$ norm error.
\emph{Right-panel}: DG norm error.}
 \label{figure:u1-32}
\end{figure}
\medskip

The convergence rate of the DG scheme under~$\p$ refinement
in the $L^2$-norm error is of order $\mathcal{O}(\p^{-(1/2+\alpha)})$, which is optimal.
The convergence rate in the DG norm is also of order $\mathcal{O}(\p^{-(1/2+\alpha)})$,
i.e., we observe {convergence rates
better than those proven in theory}.

{For the case of 32 triangular elements},
the singularity lies at the interface of several mesh cell elements.
Therefore, the usual doubling of the convergence takes place,
see, e.g., \cite{Babuvska-Suri:1987b},
and the convergence rate is $\mathcal{O}(\p^{-2\alpha})$ in both norms.

\paragraph*{Test case~2.}
As in~\cite[Example~1]{Houston-Schwab-Suli:2002},
we consider the convection field~$\betabold =(2-y^2,2-x)$,
the reaction coefficient~$c = 1+(1+x)(1+y^2)$,
and the exact solution given by~$u$ in~\eqref{u1}.
The right-hand side~$f$ is computed accordingly.
The two meshes in Figure~\ref{figure:meshes}
do not satisfy the mesh assumptions in Section~\ref{subsection:meshes}
for  the given convection field~$\betabold$.
The regularity of the exact solution~$u$ is already discussed in Test case~1.

We report the results in Figure~\ref{figure:u2-50} and~\ref{figure:u2-32}
for the~$50$ and~$32$ triangles cases, respectively.
We pick~$\alpha=0.5$, $1.5$, and~$2.5$.

\begin{figure}[h]
\centering
 \includegraphics[height=2.7in,width=2.7in]{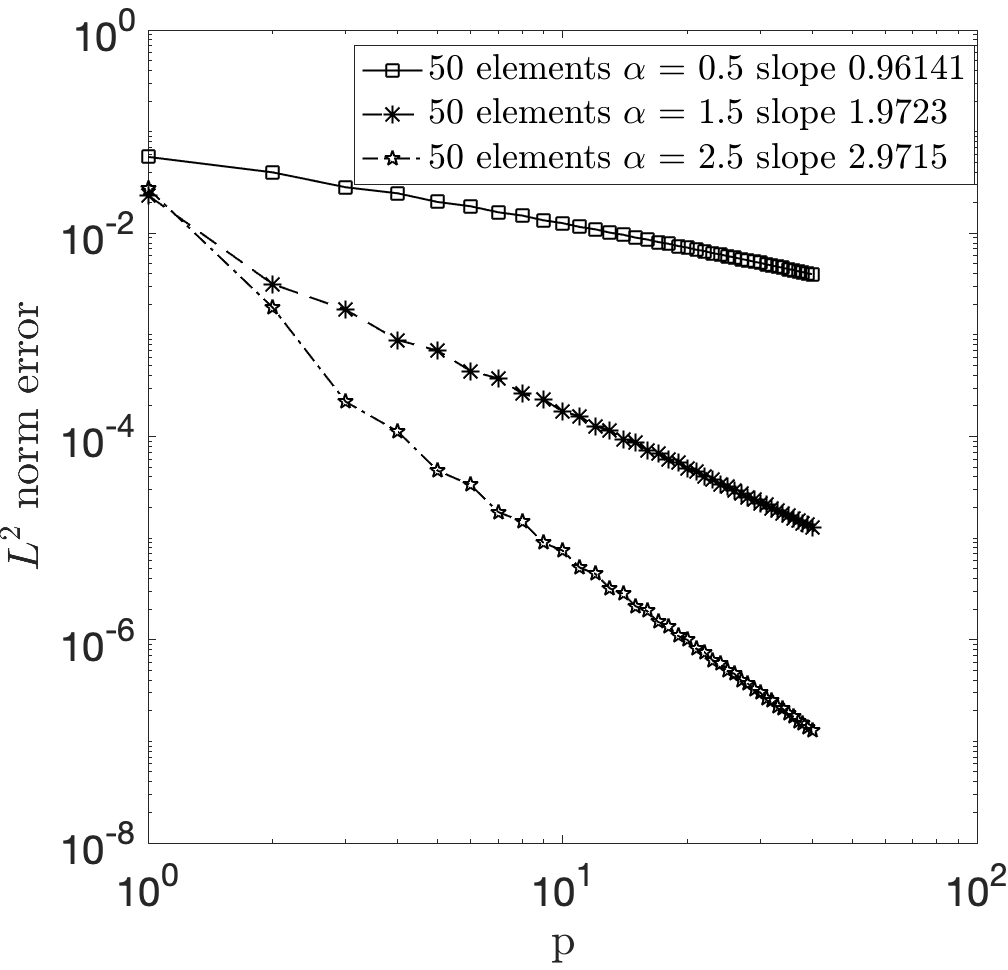}
 \qquad
 \includegraphics[height=2.7in,width=2.7in]{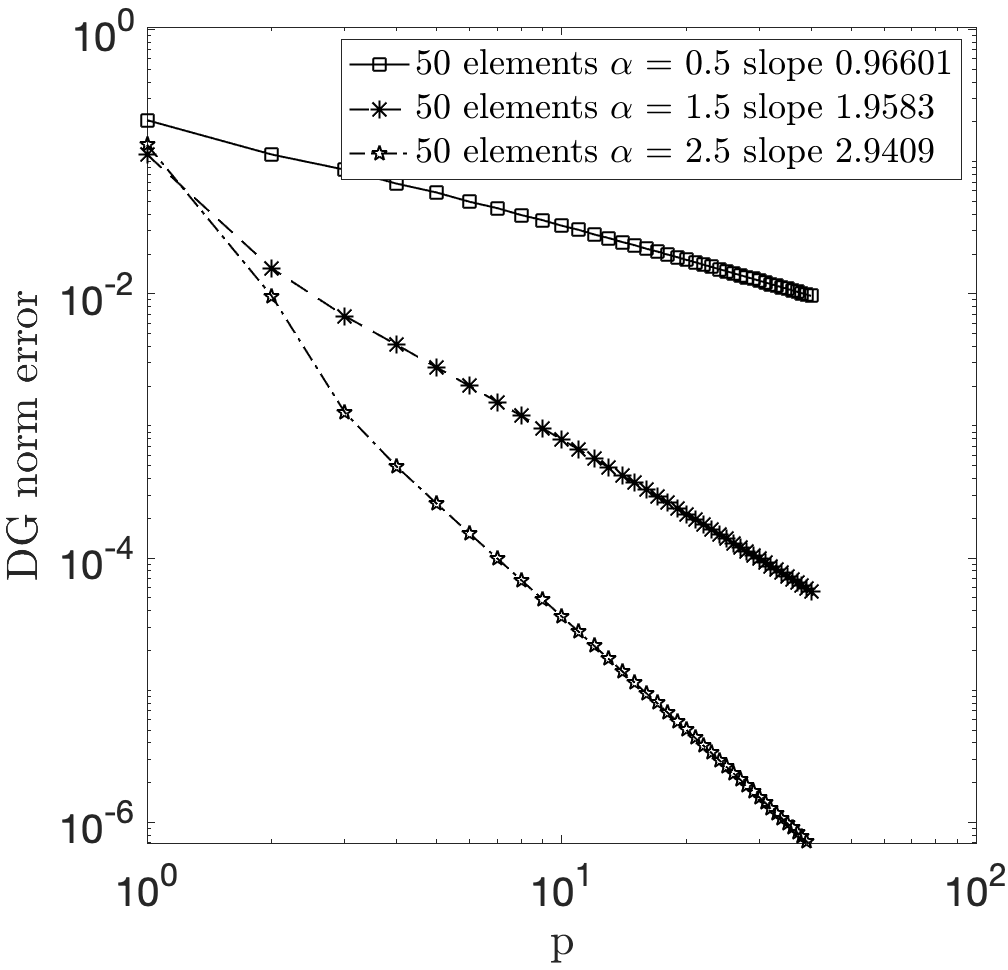}
\caption{$\p$-version for the case of variable convection field~$\betabold$.
We consider the exact solution~$u$ in~\eqref{u1}
and fix the triangular mesh with 50 elements as in Figure~\ref{figure:meshes}.
The mesh assumptions in Section~\ref{subsection:meshes} are not satisfied.
The singularity lies in the interior of a cell element.
\emph{Left-panel}: $L^2$ norm error.
\emph{Right-panel}: DG norm error.}
 \label{figure:u2-50}
\end{figure}
\begin{figure}[h]
\centering
 \includegraphics[height=2.7in,width=2.7in]{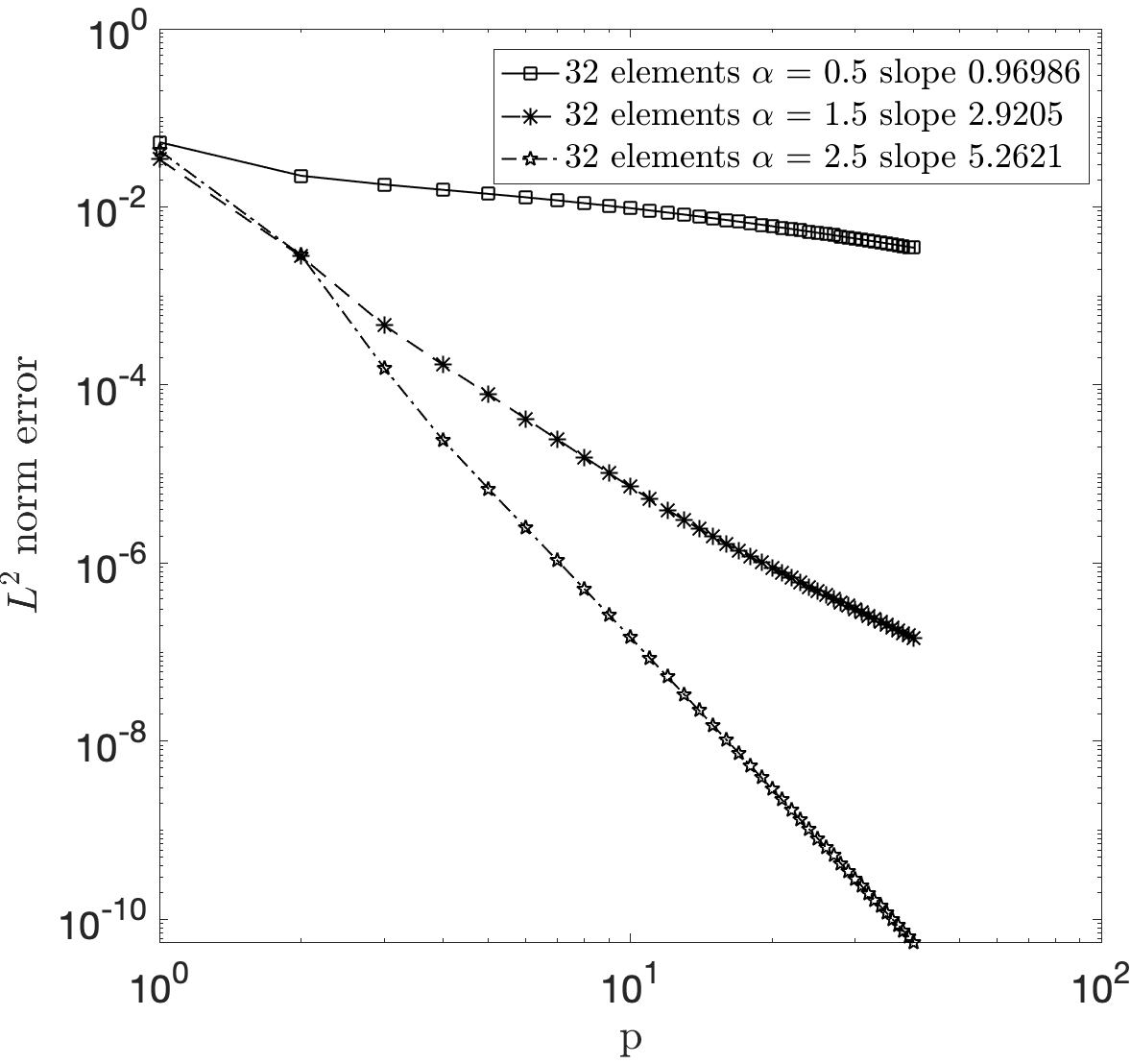}
 \qquad
 \includegraphics[height=2.7in,width=2.7in]{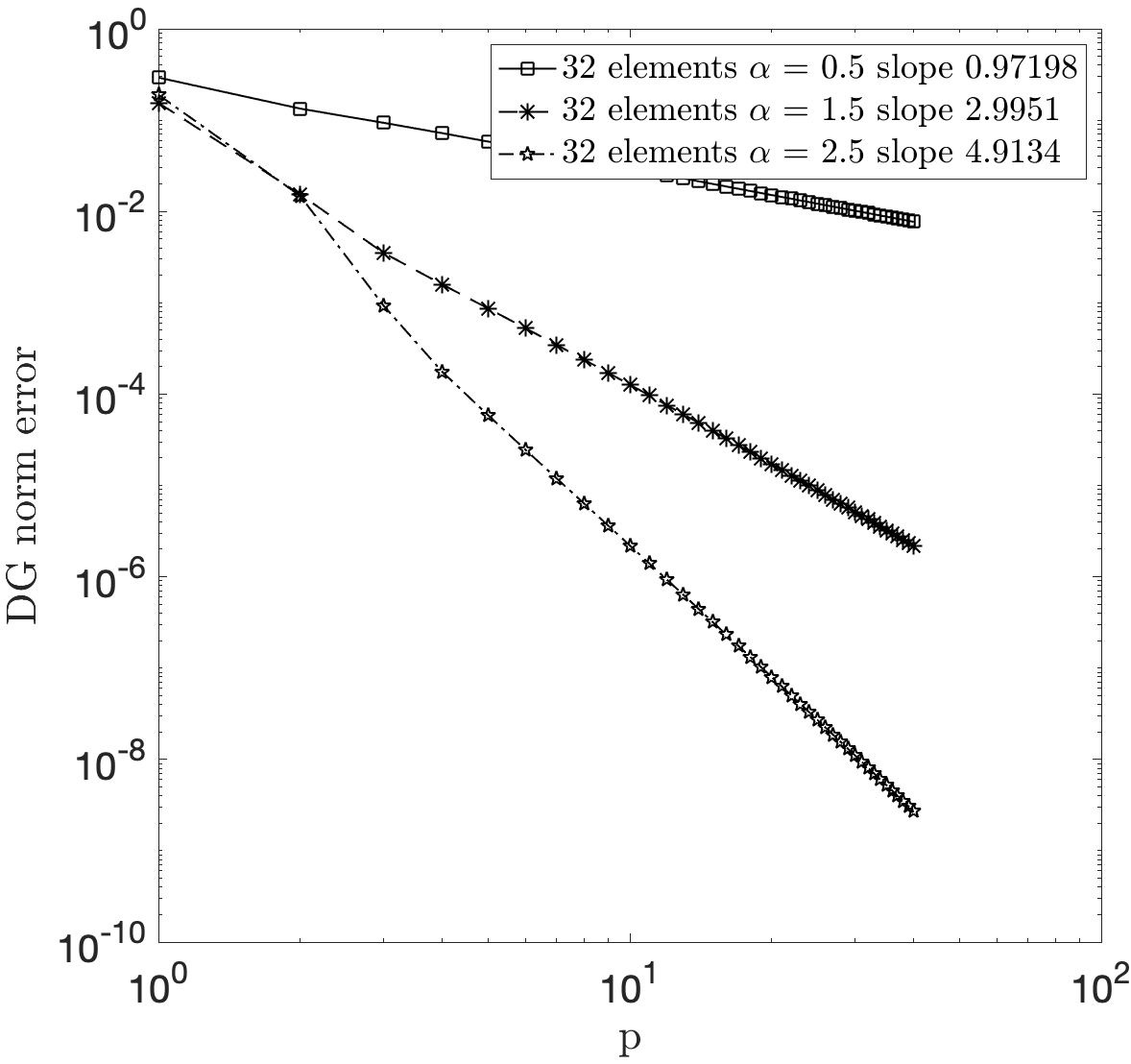}
\caption{$\p$-version for the case of variable convection field~$\betabold$.
We consider the exact solution~$u$ in~\eqref{u1}
and fix the triangular mesh with 32 elements as in Figure~\ref{figure:meshes}.
The mesh assumptions in Section~\ref{subsection:meshes} are not satisfied.
The singularity lies at the interface of several cell elements.
\emph{Left-panel}: $L^2$ norm error.
\emph{Right-panel}: DG norm error.}
\label{figure:u2-32}
\end{figure}
\medskip

We observe convergence rates as those for the constant convection field case.
This is in accordance with the numerical observation
in~\cite[Example~1]{Houston-Schwab-Suli:2002}.

\medskip

In both test cases, the use of an orthonormal basis implies
that the condition number remains moderate,
whence the convergence rates do not deteriorate for high~$\p$.

\section{Conclusions} \label{section:conclusions}
We analyzed $\p$-optimal approximation properties of the CDG operator
in the $L^2$ norm and the $L^2$ norm of the trace
on simplices with only one outflow facet in 1D, 2D, and 3D.
These results are instrumental in deriving
$\p$-optimal error estimates for the original DG method on special simplicial meshes.
Numerical experiments validate the predicted convergence rates measured in the $L^2$ norm
also for convection fields and meshes not satisfying the assumptions used in our proofs;
{convergence rates} are observed for the error measured
in the full DG norm {that are
better than those proven in theory}.

\paragraph*{Acknowledgements.}
We are very grateful to the anonymous referees
who helped us to improve the presentation and readability
of the paper.
LM is member of the Gruppo Nazionale Calcolo Scientifico-Istituto Nazionale di Alta Matematica (GNCS-INdAM)
and was partially supported by the Italian MIUR
through the PRIN Grant No. 2022-NAZ-0279.

{\footnotesize \bibliography{bibliogr}} \bibliographystyle{plain}

\end{document}